\documentclass[11pt]{amsart}
\usepackage{amsmath,amsfonts,amssymb,amsthm}
\usepackage{latexsym,bm,graphicx}
\usepackage{mathrsfs}
\usepackage{color}
\usepackage{hyperref}

\title[Asymptotic behavior of dimension]{On the asymptotic behavior of the dimension of spaces of harmonic functions with polynomial growth}
\author{Xian-Tao Huang}
\address{School of Mathematics\\  Sun Yat-sen University\\ Guangzhou 510275\\ E-mail address: hxiant@mail2.sysu.edu.cn}

\newtheorem{thm}{Theorem}[section]
\newtheorem{prop}[thm]{Proposition}
\newtheorem{lem}[thm]{Lemma}

\newtheorem{cor}[thm]{Corollary}

\theoremstyle{definition}
\theoremstyle{remark}

\newtheorem{defn}[thm]{Definition}
\newtheorem{rem}[thm]{Remark}

\numberwithin{equation}{section}

\begin{document}
%\today

\maketitle
\begin{abstract} Suppose $(M^{n},g)$ is a Riemannian manifold with nonnegative Ricci curvature, and let $h_{d}(M)$ be the dimension of the space of harmonic functions with polynomial growth of growth order at most $d$.
Colding and Minicozzi proved that $h_{d}(M)$ is finite.
Later on, there are many researches which give better estimates of $h_{d}(M)$.
We study the behavior of $h_{d}(M)$ when $d$ is large in this paper.
More precisely, suppose that $(M^{n},g)$ has maximal volume growth and has a unique tangent cone at infinity,  then when $d$ is sufficiently large, we obtain some estimates of $h_{d}(M)$ in terms of the growth order $d$, the dimension $n$ and the the asymptotic volume ratio $\alpha=\lim_{R\rightarrow\infty}\frac{\mathrm{Vol}(B_{p}(R))}{R^{n}}$.
When $\alpha=\omega_{n}$, i.e., $(M^{n},g)$ is isometric to the Euclidean space, the asymptotic behavior obtained in this paper recovers a well-known asymptotic property of $h_{d}(\mathbb{R}^{n})$.
%Our proof is basd on a combination of Li and Wang's work, Cheeger-Colding's theory and the $\mathrm{RCD}$ theory.

\vspace*{5pt}

\vspace*{5pt}

\end{abstract}
%\tableofcontents
%\setcounter{tocdepth}{1}
\section{Introduction}  %注入一级标题示例

Suppose $(M^{n},g)$ is a noncompact complete $n$-dimensional manifold. (Throughout this paper, we always assume $n\geq2$.)
We fix a point $p\in M$, and denote by $\rho(x)=d(x,p)$.
For a $d>0$, we consider the linear space:
$$\mathcal{H}_{d}(M)=\{u\in C^{\infty}(M)\mid \Delta u=0, u(p)=0, |u(x)|\leq K(\rho(x)^{d}+1) \text{ for some }K\}$$
and denote by $h_{d}(M)= \textmd{dim}\mathcal{H}_{d}(M)$.

There are many researches on polynomial growth harmonic functions on manifolds with nonnegative Ricci curvature.

Yau \cite{Y75} proved that any positive harmonic function on complete manifolds with nonnegative Ricci curvature is constant.
In \cite{Cg82}, Cheng further proved that on such manifolds any harmonic function of sublinear growth must be constant.

On the other hand, it is well-known that
$h_{d}(\mathbb{R}^{n})\thicksim\frac{2}{(n-1)!}d^{n-1}$
as $d\rightarrow\infty$.

In view of these results, Yau conjectured that on a manifold $(M^{n},g)$ with non-negative Ricci curvature it always holds $h_{d}(M)<\infty$.

Li and Tam solved the case for $d=1$ and the case $n=2$ of this conjecture (see \cite{LT89} \cite{LT91}).
Yau's conjecture was completely solved by Colding and Minicozzi in \cite{CM97a}.
Later on, there are many researches which give better estimates of $h_{d}(M)$.

For example, a more precise upper bound for the dimension was obtained:

\begin{thm}[see \cite{CM98b}, \cite{L97}]\label{upperbound}
If $(M^{n},g)$ has nonnegative Ricci curvature, then there exists a constant $C=C(n)>0$ such that $h_{d}(M)\leq Cd^{n-1}$ for all $d\geq1$.
\end{thm}

Note that the power $n-1$ in Theorem \ref{upperbound} is sharp compared to the Euclidean case.

The example given by Donnelly (see \cite{Don01}) shows that we can not expect $h_{d}(M^{n})\leq h_{d}(\mathbb{R}^{n})$ holds for every $d$ and every $(M^{n},g)$ with nonnegative Ricci curvature.

On the other hand, in Colding and Minicozzi's paper \cite{CM98b}, they proved that if $(M^{n},g)$ has nonnegative Ricci curvature then
$$h_{d}(M)\leq C_{1} \alpha d^{n-1}+C_{2}f(d^{n-1}),$$
where $\alpha$ is the  asymptotic volume ratio; $C_{1}$ and $C_{2}$ are positive constants depending only on $n$; the function $f:\mathbb{R}^{+}\rightarrow\mathbb{R}^{+}$ also depends only on $n$ and satisfies
$f(t)\leq t$ and $\lim_{t\rightarrow\infty}\frac{f(t)}{t}=0$.
See Theorem 0.26 in \cite{CM98b}.

In the case that $(M^{n},g)$ has nonnegative sectional curvature, Li and Wang proved that $\lim_{d\rightarrow\infty} d^{-n}\sum_{i=1}^{d}h_{i-1}(M)\leq\frac{2\alpha}{n!\omega_{n}}$ and $\liminf_{d\rightarrow\infty}d^{1-n}h_{d}(M)\leq\frac{2\alpha}{(n-1)!\omega_{n}}$, where $\alpha$ is the  asymptotic volume ratio, $\omega_{n}$ is the volume of a unit ball in the $n$-dimensional Euclidean space.
See Theorem 2.2 and Corollary 2.3 of \cite{LW99}.
Note that in the case that $\alpha=\omega_{n}$, i.e. $(M^{n},g)$ is isometric to the Euclidean space, the above two inequalities are in fact equalities.

From the theorems in \cite{CM98b} and \cite{LW99}, it seems that the behavior of $h_{d}(M)$ becomes better when $d$ is large, and the better bound of $h_{d}(M)$ has some relations to the asymptotic volume ratio $\alpha$.

In this paper, we also study asymptotic properties of $h_{d}(M)$ when $M$ has nonnegative Ricci curvature.
The following theorem is our main result.

\begin{thm}\label{main-4}
Suppose $(M^{n}, g)$ is a complete manifold with nonnegative Ricci curvature and maximal volume growth, i.e. there is a constant $\alpha\in(0,\omega_{n}]$ such that $$\alpha=\lim_{r\rightarrow\infty}\frac{\mathrm{Vol}(B_{p}(r))}{r^{n}}.$$
Assume that $M$ has a unique tangent cone at infinity.
Denote by $h_{d} = \textmd{dim}\mathcal{H}_{d}(M)$.
Then we have
\begin{align}\label{1.22221}
\lim_{d\rightarrow\infty} d^{-n}\sum_{i=1}^{d}h_{i-1}=\frac{2\alpha}{n!\omega_{n}},
\end{align}
and
\begin{align}\label{1.22222}
\liminf_{d\rightarrow\infty}d^{1-n}h_{d}=\frac{2\alpha}{(n-1)!\omega_{n}}.
\end{align}
\end{thm}

As a corollary, we have
\begin{cor}\label{cor-1}
Let $(M^{n}, g)$ be a complete manifold with nonnegative Ricci curvature.
Suppose that $(M^{n}, g)$ has maximal volume growth and has a unique tangent cone at infinity, then the following three conditions are equivalent:
\begin{enumerate}
  \item $\lim_{d\rightarrow\infty} d^{-n}\sum_{i=1}^{d}h_{i-1}=\frac{2}{n!}$;
  \item $\lim_{d\rightarrow\infty}d^{1-n}h_{d}=\frac{2}{(n-1)!}$;
  \item $(M^{n}, g)$ is isometric to the Euclidean space.
\end{enumerate}
\end{cor}

Theorem \ref{main-4} improves the results in Theorem 2.2 and Corollary 2.3 of \cite{LW99} in two aspects.
Firstly, the assumption that the manifold has nonnegative sectional curvature has been weakened.
Secondly, we obtain equalities in (\ref{1.22221}) and (\ref{1.22222}), while in \cite{LW99}, the conclusions are some inequalities.

The proof in \cite{LW99} considers the level sets of the distance function $d(p,\cdot)$.
The nonnegative sectional curvature assumption and the smooth structure are used to control some quantities such as the eigenvalues of the level sets and the restriction of harmonic functions on the level sets.
In addition, Weyl's law on smooth manifolds is used in \cite{LW99}.

In the present paper, we combine Li and Wang's proof (\cite{LW99}), Cheeger-Colding's theory (see e.g. \cite{CC96}, \cite{CC97}) and the recent development in \textmd{RCD} theory to prove Theorem \ref{main-4}.

If $(M^{n},g)$ is a complete manifold with nonnegative Ricci curvature and maximal volume growth, from Cheeger and Colding's theory, every tangent cone at infinity of $M^{n}$ is a metric cone.
However, in general, the tangent cones may be not unique; they depend on the sequences of scales used to blow down the manifold, see \cite{P97}, \cite{CC97} for examples.
For manifolds with nonnegative sectional curvature, the tangent cone at infinity is unique.
There are other sufficient conditions to ensure uniqueness of tangent cone for manifold with nonnegative Ricci curvature and maximal volume growth, see e.g. \cite{CM14}.

For a manifold $(M^{n},g)$ with nonnegative Ricci curvature and maximal volume growth, its asymptotically conic property makes the harmonic functions with polynomial growth on it behave like harmonic functions with polynomial growth on a cone.
This idea has impact on many works, see \cite{CM97b}, \cite{CM98b}, \cite{D04}, \cite{X15}, \cite{H15}, \cite{H16} etc.

Let's mention that, in \cite{D04}, for a manifold $(M^{n},g)$ with nonnegative Ricci curvature, maximal volume growth and unique tangent cone at infinity, it is proved that $h_{d}(M)\geq Cd^{n-1}$, where the constant $C$ depends on $n$ and the asymptotic volume ratio.
See also \cite{X15} and \cite{H16} for related researches.

Before we going on introducing the main steps of our proof, let's give a brief introduction on the study of \textmd{RCD} theory.

Using the theory of optimal transport, Lott, Villani (\cite{LV09}) and Sturm (\cite{St06I} \cite{St06II}) independently introduced the $\textmd{CD}(K,N)$-condition, which is a notion of `Ricci bounded from below by $K\in \mathbb{R}$ and dimension bounded above by $N\in[1,\infty]$' for general metric measure spaces.
Later on, Bacher and Sturm (\cite{BS10}) introduced reduced curvature-dimension condition $\textmd{CD}^{\ast}(K,N)$.
For $N\in[1,\infty)$, both the $\textmd{CD}(K,N)$ condition and $\textmd{CD}^{\ast}(K,N)$ condition are compatible with the Riemannian case.
These two conditions are both stable under the measured Gromov-Hausdorff convergence.
In particular, measured Gromov-Hausdorff of manifolds with a uniform lower Ricci curvature bound (`Ricci-limit spaces' for short) are both $\textmd{CD}(K,N)$ spaces and $\textmd{CD}^{\ast}(K,N)$ spaces.
However, both the class of $\textmd{CD}(K,N)$ spaces and the class of $\textmd{CD}^{\ast}(K,N)$ spaces are too large in the sense that they both include Finsler geometries.

Ambrosio, Gigli and Savar\'{e} (\cite{AGS14}) introduced the notion of $\textmd{RCD}(K,\infty)$ spaces (see also \cite{AGMR15} for the simplified axiomatization), which rules out Finsler geometries, while retaining the stability properties under measured Gromov-Hausdorff convergence.
Later on, many researchers began to considered the $\textmd{RCD}^{*}(K,N)$ spaces (for the case of $N<\infty$), see \cite{AMS15} \cite{EKS15} \cite{Gig13} etc.
To be precise, a $\textmd{RCD}^{*}(K,N)$ space $(X,d,m)$ is a $\textmd{CD}^{*}(K,N)$ space such that the Sobolev space $W^{1,2}(X)$ is a Hilbert space.

In the recent years there are too many important developments on the $\textmd{RCD}^{*}$ theory so that we cannot list them here.
Comparing to the properties hold only on manifolds or Ricci-limit spaces, the results obtained on general $\textmd{RCD}^{*}(K,N)$ spaces are `intrinsic' in the sense that their proofs do not depend on the smooth structure.
In addition, many recent results on $\textmd{RCD}^{*}(K,N)$ spaces supplement the knowledge on manifolds or Ricci-limit spaces.

On the other hand, there are also many results in Cheeger-Colding's theory which seem hard to be generalized to general $\textmd{RCD}^{*}(K,N)$ spaces.
For example, it is unknown whether the `almost volume cone implies almost metric cone' properties (see \cite{CC96}) and the structure theory for Gromov-Hausdorff limit of non-collapsed manifolds with a uniform lower Ricci curvature bound (see e.g. \cite{CC97}) hold on general $\textmd{RCD}^{*}(K,N)$ spaces.

Now we continuous to introduce the main ideas in the proof of Theorem \ref{main-4}.

Suppose $(M^{n},g)$ is a manifold with nonnegative Ricci curvature and maximal volume growth, from the theory in \cite{CC96} \cite{CC97}, if $r_{i}\rightarrow\infty$, then up to a subsequence, $(M, p, g_{i}=r_{i}^{-2}g)$  convergence in the measured Gromov-Hausdorff sense to some metric measure space $(M_{\infty}, p_{\infty}, \rho_{\infty},\nu_{\infty})$ such that $(M_{\infty}, p_{\infty}, \rho_{\infty})$ is a metric cone $(C(X),p_{\infty},d_{C(X)})$, and $\nu_{\infty}$ is a multiple of the Hausdorff measure $\mathcal{H}^{n}$.

It is unclear whether the induced metric measure structure on $X$, denoted by $(X,d_{X},m_{X})$, is a Ricci-limit space.
However, from \cite{Ke15}, we know $(X,d_{X},m_{X})$ is a $\textmd{RCD}^{*}(n-2,n-1)$ space.
This enable us to apply the $\textmd{RCD}$ theory to $(X,d_{X},m_{X})$.
More precisely, in a recent work \cite{AHT17}, Ambrosio, Honda and Tewodrose find a sufficient and necessary condition for Weyl's law for eigenvalues of Neumann problem on a $\textmd{RCD}^{*}(K,N)$ space.
The proof of Weyl's law in \cite{AHT17} is mainly based on the structure theory of $\textmd{RCD}^{*}(K,N)$ spaces (see \cite{MN14} \cite{KM16} \cite{DePMR16} \cite{GP16}) as well as the study of short-time behavior of the heat kernel on $\textmd{RCD}^{*}(K,N)$ spaces.
By the structure theory of Cheeger and Colding on $C(X)$, we know $(X,d_{X},m_{X})$ satisfies the criterion in \cite{AHT17}, hence Weyl's law holds on $(X,d_{X},m_{X})$.
See Section \ref{s2} for details.

On the other hand, on the conic space $(C(X),d_{C(X)},\nu_{\infty})$, any harmonic function has a very beautiful form, and the eigenfunctions for Neumann problem on $(X,d_{X},m_{X})$ are closely related to the harmonic functions on $C(X)$.
Furthermore, the sum of roots of eigenvalues for Neumann problem on $(X,d_{X},m_{X})$ has an interesting property, see Proposition \ref{prop3.6}.
These enable us to make use of Weyl's law on $(X,d_{X},m_{X})$ to obtain information on harmonic function on $C(X)$.
See Sections \ref{s3} and \ref{s5} for details.

Thus, in order to finish the proof of Theorem \ref{main-4}, the remaining problem is to choose suitable scales to blow down the manifold so that at the same time some information on $\mathcal{H}_{d}(M)$ can pass to the limit.
This is done in Section \ref{s4}.
We remark that in this step, we use some results on convergence of functions (especially harmonic functions) defined on domain of different spaces.
When these spaces are manifolds with a uniform Ricci curvature lower bound, some results are obtained in \cite{D02}, \cite{Xu14}, \cite{H11} etc.
The more recent researches in \cite{GMS15}, \cite{ZZ17}, \cite{AH17} etc. show that similar results holds when these spaces are general $\mathrm{RCD}^{*}(K,N)$ spaces.

Combining the arguments from Sections \ref{s2} to \ref{s5} we can finish the proof of Theorem \ref{main-4}.

Finally, we remark that the assumption on the uniqueness of tangent cone at infinity is because of technical reasons.
See Remark \ref{rem5.2} for discussions.

\vspace*{20pt}

\noindent\textbf{Acknowledgments.} %The author would like to express his gratitude to Guoyi Xu for several helpful conversations and discussions.
The author would like to thank Prof. B.-L. Chen, H.-C. Zhang and X.-P. Zhu and Dr. Y. Jiang for helpful discussions.
The author is partially supported by NSFC 11521101.

\section{Properties of tangent cones at infinity}\label{s2}

Following Definition 5.1 in \cite{Ke15}, we use the terminology of $(0, N)$-cones throughout this paper :
\begin{defn}[$(0, N)$-cones]
Suppose $(X, d_{X}, m_{X})$ is a metric measure space with $\textmd{diam}(X)\leq\pi$, the $(0, N)$-cone over $(X, d_{X}, m_{X})$ is a metric measure space $(C(X),d_{C(X)},m_{C(X)})$ such that:
\begin{itemize}
  \item $C(X)=X\times[0,\infty)/(X\times\{0\})$;
  \item $(C(X),d_{C(X)})$ is a metric cone with
  $$d_{C(X)}((x_{1},t_{1}),(x_{2},t_{2}))=\sqrt{t_{1}^{2}+t_{2}^{2}-2t_{1}t_{2}\cos(d_{X}(x_{1},x_{2}))}$$
  for $(x_{1},t_{1}),(x_{2},t_{2})\in C(X)$;
  \item $m_{C(X)}=t^{N}dt\otimes m_{X}$.
\end{itemize}
\end{defn}

Obviously, the map $\Phi:X\times(0,\infty)\rightarrow C(X)$ is locally bi-Lipschitz.
More precisely, given two positive numbers $s<r$, $\Phi|_{X\times(r-s,r+s)}$ is a $(1\pm\Psi(s;r))$-bi-Lipschitz map onto its image, here $X\times(r-s,r+s)$ is equipped with the product metric, and $\Psi(s;r)$ means a nonnegative function such that when $r$ is fixed, we have $\lim_{s\rightarrow0}\Psi(s;r)=0.$

For the definition and a quick introduction of properties of $\mathrm{RCD}^{*}(K,N)$ spaces, there are many references, see e.g. Section 2 in \cite{MN14}.

Let $(X_{i},x_{i},d_{i},m_{i})$, $i\in\mathbb{N}$, be a sequence of pointed metric measure spaces, we say $(X_{i},x_{i},d_{i},m_{i})$ converges to $(X_{\infty},x_{\infty},d_{\infty},m_{\infty})$ in the pointed measured Gromov-Hausdorff (pmGH for short) sense if there are sequences $R_{i}\uparrow+\infty$, $\epsilon_{i}\downarrow0$ and Borel maps $\varphi_{i}:X_{i}\rightarrow X_{\infty}$ such that
\begin{enumerate}
  \item $\varphi_{i}(x_{i})=x_{\infty}$;
  \item $|d_{i}(x, y)-d_{\infty}(\varphi_{i}(x),\varphi_{i}(y))| <\epsilon_{i}$ for any $i$ and all $x, y \in B_{x_{i}}(R_{i})$;
  \item the $\epsilon_{i}$-neighbourhood of $\varphi_{i}(B_{x_{i}}(R_{i}))$ contains $B_{x_{\infty}}(R_{i}-\epsilon_{i})$;
  \item for any $f\in C_{b}(X_{\infty})$ with bounded support it holds $\lim_{n\rightarrow\infty}\int f\circ\varphi_{i}dm_{i}=\int fdm_{\infty}$.
\end{enumerate}

When all the $(X_{i},x_{i},d_{i},m_{i})$ satisfy a uniformly doubling condition, the pmGH convergence is equivalent to the pointed measured Gromov convergence introduced in \cite{GMS15}.
See \cite{GMS15} for a detailed study of these notions.

Suppose $(M^{n},g)$ is a complete Riemannian manifold with nonnegative Ricci curvature.
Let $\rho$ be the distance determined by $g$, $\mu$ be the volume element determined by $g$.
We can define $(M_{i}, p, \rho_{i},\nu_{i})$, where $M_{i}$ is the same differential manifold as $M^{n}$, $p$ is a fixed point on $M_{i}=M^{n}$, $\rho_{i}$ is the distance determined by the rescaled metric $g_{i}=r_{i}^{-2}g$, $r_{i}\rightarrow\infty$,  and $\nu_{i}$ is the renormalized measure defined by
$$\nu_{i}(A):=\frac{1}{\mu_{i}(B^{(i)}_{p}(1))}\mu_{i}(A),$$
where $A\subset M_{i}$, and $\mu_{i}$ is the volume element determined by $g_{i}$.
Then by Gromov's compactness theorem and Theorem 1.6 in \cite{CC97}, up to a subsequence, $(M_{i}, p, \rho_{i},\nu_{i})$ converges to some $(M_{\infty}, p_{\infty}, \rho_{\infty},\nu_{\infty})$ in the pmGH sense.
$(M_{\infty},p_{\infty},\rho_{\infty})$ is called a tangent cone at infinity of $M$.
In general, the tangent cones at infinity may not unique.

Suppose in addition $(M^{n},g)$ has maximal volume growth:
\begin{align}\label{1.01111}
\alpha:=\lim_{r\rightarrow\infty}\frac{\mu(B_{p}(r))}{r^{n}}>0.
\end{align}
From the theory in \cite{CC96} \cite{CC97} etc., in this case $(M_{\infty}, p_{\infty}, \rho_{\infty},\nu_{\infty})$ has many strong properties.

Firstly, by Theorem 7.6 in \cite{CC96}, $(M_{\infty}, p_{\infty}, \rho_{\infty})$ is a metric cone, i.e. $(M_{\infty}, p_{\infty}, \rho_{\infty})=(C(X),p_{\infty},d_{C(X)})$, where $(X,d_{X})$ is a compact metric space with $\textmd{diam}(X)\leq\pi$.

Secondly, by Theorem 5.9 in \cite{CC97}, $dim_{\mathcal{H}} M_{\infty} = n$ (in this paper $dim_{\mathcal{H}}$ means the Hausdorff dimension), and for any $R > 0$ and $M_{i}\ni q_{i}\xrightarrow{d_{GH}} z\in M_{\infty}$, we have
    \begin{align}\label{1.11111}
        \lim_{i\rightarrow\infty} \mu_{i}(B^{(i)}_{q_{i}}(R))=\mathcal{H}^{n}(B_{z}(R)).
    \end{align}

By (\ref{1.01111}) and (\ref{1.11111}) we have
\begin{align}\label{1.11112}
\mathcal{H}^{n}(B_{p_{\infty}}(1))=\lim_{i\rightarrow\infty} \mu_{i}(B^{(i)}_{p}(1))=\lim_{r_{i}\rightarrow\infty}\frac{\mu(B_{p}(r_{i}))}{r_{i}^{n}}=\alpha,
\end{align}
and the limit renormalized measure satisfies $\nu_{\infty}=\frac{1}{\alpha}\mathcal{H}^{n}$.

Note that $dim_{\mathcal{H}} X=n-1$.
Furthermore, $\mathcal{H}^{n}$ satisfies a co-area formula on $C(X)$, that is, for any $\Omega\subset\subset C(X)$,
\begin{align}\label{conic3}
\mathcal{H}^{n}(\Omega)=\int_{0}^{\infty}s^{n-1}ds\int_{X}\chi(\Omega_{s})d\mathcal{H}^{n-1},
\end{align}
where $\Omega_{s}=\{x\in X\mid z=(x,s)\in\Omega\}$.
See Proposition 7.6 in \cite{H15} for a proof.
By (\ref{1.11112}) and (\ref{conic3}), it is easy to see
\begin{align}\label{1.11115}
\mathcal{H}^{n-1}(X)=n\alpha.
\end{align}

Define the measure $m_{X}:=\frac{1}{\alpha}\mathcal{H}^{n-1}$ on $X$, then it is easy to see
$(M_{\infty}, p_{\infty}, \rho_{\infty},\nu_{\infty})$ is a $(0, n-1)$-cone over $(X, d_{X}, m_{X})$.

In the next, we use the $\textmd{RCD}$ theory to study $(M_{\infty}, p_{\infty}, \rho_{\infty},\nu_{\infty})=(C(X),d_{C(X)},m_{C(X)})$ and $(X, d_{X}, m_{X})$.

Obviously, $(M_{\infty}, p_{\infty}, \rho_{\infty},\nu_{\infty})$ is an $\textmd{RCD}^{*}(0, n)$ space.
Thus by Corollary 1.3 of \cite{Ke15}, we further know that the cross section $(X, d_{X}, m_{X})$ is an $\textmd{RCD}^{*}(n-2,n-1)$ space.
In particular, we know $(X, d_{X}, m_{X})$ satisfies a volume doubling property and Poincar\'{e} inequality (in fact, to derive these two properties on $(X, d_{X}, m_{X})$, we can also use the arguments in \cite{D02} to avoid the use of the $\textmd{RCD}$ property).

Let $\Delta_{X}$ be the Laplacian operator on $(X,d_{X},m_{X})$.
Let $0=\lambda_{0}<\lambda_{1}\leq\lambda_{2}\leq\ldots$  be the eigenvalues of $\Delta_{X}$, $\{\varphi_{i}(x)\}_{i=0}^{\infty}$ be the corresponding eigenfunctions, i.e.
\begin{align}
-\Delta_{X}\varphi_{i}(x)=\lambda_{i}\varphi_{i}(x).
\end{align}
We require $\int_{X}|\varphi_{i}|^{2}dm_{X}=1$, and $\int_{X}\varphi_{i}\varphi_{j}dm_{X}=0$ for every $i\neq j$.
Note that $\lambda_{i}\rightarrow\infty$ and $\{\varphi_{i}(x)\}_{i=0}^{\infty}$ spans $L^{2}(X,m_{X})$, which can be derived from a standard argument basing on a Rellich-type Compactness Theorem.
Every $\varphi_{i}$ always has a Lipschitz representative in the corresponding Sobolev class (see \cite{J14}).
In the remaining part of this paper, the $\varphi_{i}$'s are always required to be Lipschitz.

Recently, Weyl's law on $\textmd{RCD}^{*}(K,N)$ spaces has been studied.
In \cite{AHT17}, Ambrosio, Honda and Tewodrose gives a sufficient and necessary condition for the Weyl's law for eigenvalues of Neumann problem on compact $\textmd{RCD}^{*}(K,N)$ spaces, see Theorem 4.3 in \cite{AHT17}.
(See \cite{ZZ17} for Weyl's law for eigenvalues of Dirichlet problem on $\textmd{RCD}^{*}(K,N)$ spaces.)
Before we give a statement of Weyl's law, we first introduce some notion and background knowledge.

\begin{defn}\label{def_reg_set}
The $k$-dimensional regular set $\mathcal{R}_{k}$ of an $\textmd{RCD}^{*}(K,N)$ space $(X,d,m)$ is the set of points $x\in\mathrm{supp}m$ such that
$$(X,r^{-1}d,\frac{1}{m(B_{x}(r))}m,x)\xrightarrow{pmGH} (\mathbb{R}^{k},d_{\mathbb{R}^{k}},\frac{1}{\omega_{k}}\mathcal{H}^{k},0_{k})$$
as $r\rightarrow0^{+}$.
\end{defn}

We remark that the definition of the regular set in Definition \ref{def_reg_set} is a bit different from that of  Mondino and Naber (see \cite{MN14}).
However, the two definitions are equivalent, see Remark 3.5 in \cite{AHT17}.
Thus, by \cite{MN14}, the following structure theorem holds.
\begin{thm}[see \cite{MN14}]\label{structure}
Let $(X,d,m)$ be an $\mathrm{RCD}^{*}(K,N)$ space with $K\in\mathbb{R}$ and $N\in (1,\infty)$. Then
$$m(X\setminus\bigcup_{k=1}^{[N]}\mathcal{R}_{k})= 0.$$
Moreover, for any sufficiently small $\epsilon$, we can cover each $\mathcal{R}_{k}$, up to an $m$-negligible subset, by a countable collection of sets $\{U_{\epsilon}^{k,l}\}_{l}$ such that each $U_{\epsilon}^{k,l}$ is $(1+\epsilon)$-biLipschitz to a subset of $\mathbb{R}^{k}$.
\end{thm}

As in \cite{AHT17}, we have the following definitions.
\begin{defn}
Suppose $(X,d,m)$ is an $\textmd{RCD}^{*}(K,N)$ space.
$\mathrm{dim}_{d,m}(X)$ is defined to be the largest integer $k$ such that $\mathcal{R}_{k}$ has positive $m$-measure.
For any positive integer $k$, $\mathcal{R}^{*}_{k}\subset\mathcal{R}_{k}$ is defined to be
$$\mathcal{R}^{*}_{k}:=\bigl\{x\in \mathcal{R}_{k}\bigl| \exists \lim_{r\rightarrow0^{+}}\frac{m(B_{x}(r))}{\omega_{k}r^{k}}\in(0,\infty)\bigr\}.$$
\end{defn}

On a compact $\textmd{RCD}^{*}(K,N)$ space $(X,d,m)$, we consider the Neumann eigenvalue problems, and let $0=\lambda_{0}<\lambda_{1}\leq\lambda_{2}\leq\ldots$ be the eigenvalues counted with multiplicity.
Let
\begin{align}
N_{(X,d,m)}(\lambda) := \#\{i \in\mathbb{N}^{+} |\lambda_{i}\leq\lambda\}\nonumber
\end{align}
be the counting function.

The following theorem is a special case of Weyl's law obtained from \cite{AHT17}:

\begin{thm}[see Corollary 4.4 in \cite{AHT17}]\label{weyl_law}
Let $(X,d,m)$ be a compact $\mathrm{RCD}^{*}(K,N)$ space with $K\in\mathbb{R}$ and $N\in (1,\infty)$, and let $k=\mathrm{dim}_{d,m}X$.
Suppose $k=N$, then we have
\begin{align}\label{1.11116}
\lim_{\lambda\rightarrow\infty}\frac{N_{(X,d,m)}(\lambda)}{\lambda^{\frac{k}{2}}} =\frac{\omega_{k}}{(2\pi)^{k}}\mathcal{H}^{k}(\mathcal{R}^{*}_{k})<\infty.
\end{align}
\end{thm}

In the following, we will apply Theorem \ref{weyl_law} to the $\textmd{RCD}^{*}(n-2,n-1)$ space $(X,d_{X},m_{X})$, where $X$ is the cross section of a tangent cone at infinity of $(M,g)$ as in the beginning of this section.
In view of Weyl's law in the form of (\ref{1.11116}), the most important point is to understand $k=\mathrm{dim}_{d_{X},m_{X}}X$ and $\mathcal{R}^{*}_{k}$.

We will apply the structure theory of Cheeger and Colding on the Ricci-limit space $(M_{\infty}, p_{\infty}, \rho_{\infty},\nu_{\infty})=(C(X),d_{C(X)},m_{C(X)})$ to get information on the structure of $(X,d_{X},m_{X})$.

Recall that in Definition 0.1 of \cite{CC97}, a point $y\in C(X)$ is called $k$-regular if every tangent cone at $y$ is isometric to $\mathbb{R}^{k}$.
Let $\tilde{\mathcal{R}}_{k}$ be the set of $k$-regular points on $C(X)$.
Here we note that even though in the definition of \cite{CC97}, at a $k$-regular points, the limit measure on $\mathbb{R}^{k}$ is not stated, but from the arguments in Proposition 1.35 of \cite{CC97}, we can easily prove that the limit measure on the tangent cone $\mathbb{R}^{k}$ is just a multiple of the standard Hausdorff measure.
Thus when restricted to the Ricci-limit spaces, the $k$-regular points in the sense of \cite{MN14} (see Definition \ref{def_reg_set}) coincides with the ones in \cite{CC97}.

\begin{prop}\label{prop2.1111}
Suppose $(C(X),d_{C(X)},m_{C(X)})$ is a tangent cone at infinity of a manifold $(M^{n},g)$ with nonnegative Ricci curvature satisfying (\ref{1.01111}).
Then for the $\mathrm{RCD}^{*}(n-2,n-1)$ space $(X, d_{X}, m_{X})$, we have $\mathrm{dim}_{d_{X},m_{X}}X=n-1$ and $\mathcal{H}^{n-1}(X\setminus\mathcal{R}_{n-1}^{*})=0$.
In particular, the following Weyl's law holds:
\begin{align}\label{1.11117}
\lim_{\lambda\rightarrow\infty}\frac{N_{(X,d_{X},m_{X})}(\lambda)}{\lambda^{\frac{n-1}{2}}} =\frac{\omega_{n-1}}{(2\pi)^{n-1}}\mathcal{H}^{n-1}(X)=\frac{n\omega_{n-1}\alpha}{(2\pi)^{n-1}}.
\end{align}
\end{prop}

\begin{proof}
We use $\mathcal{R}_{k}$ to denote the set of $k$-regular points of $(X, d_{X}, m_{X})$, and use $\tilde{\mathcal{R}}_{k}$ to denote the set of $k$-regular points of $(C(X),d_{C(X)},m_{C(X)})$.

Since $(C(X),d_{C(X)},m_{C(X)})$ is the Gromov-Hausdorff limit of a sequence of non-collapsed manifolds $(M_{i}, p, \rho_{i},\nu_{i})$, from the results in \cite{CC97}, we have $m_{C(X)}(C(X)\setminus\tilde{\mathcal{R}}_{n})= 0$ (see also \cite{CN12} for a more general result which even holds in collapsed case).

Denote by $\mathcal{S}'=X\setminus\mathcal{R}_{n-1}$, then for any $x\in \mathcal{S}'$, there exists a sequence of $r_{i}\downarrow0$ such that
$$(X,r_{i}^{-1}d_{X},\frac{1}{m_{X}(B_{x}(r_{i}))}m_{X},x)\xrightarrow{pmGH}  (Y,d_{Y},\nu^{Y},y)$$
such that $(Y,d_{Y},\nu^{Y},y)$ is not isomorphic to the Euclidean space $(\mathbb{R}^{n-1},d_{\mathbb{R}^{n-1}},\frac{1}{\omega_{n-1}}\mathcal{H}^{n-1},0_{n-1})$.
%while $d_{\mathcal{M}}((Y,d_{Y},\nu^{Y},x_{\infty}),(\mathbb{R}^{n-1},d_{\mathbb{R}^{n-1}}, \frac{1}{\omega_{n-1}}\mathcal{H}^{n-1},0_{n-1})) >\delta$ for some positive constant $\delta$.
Let $x':=(x,1)\in C(X)$.
It is easy to see
$$(C(X),r_{i}^{-1}d_{C(X)},\frac{1}{m_{C(X)}(B_{x'}(r_{i}))}m_{C(X)},x')\xrightarrow{pmGH}  (Y,d_{Y},\nu^{Y},y)\otimes(\mathbb{R}^{1},d_{\mathbb{R}^{1}},\bar{c}\mathcal{H}^{1},0_{1}),$$
where $\bar{c}$ is a constant to ensure the measure in the limit measure take value $1$ on the geodesic ball of radius 1 around $(y,0_{1})$.
It is easy to check  $(Y,d_{Y},\nu^{Y},y)\otimes(\mathbb{R}^{1},d_{\mathbb{R}^{1}},\bar{c}\mathcal{H}^{1},0_{1})$ is not isomorphic to $(\mathbb{R}^{n},d_{\mathbb{R}^{n}}, \frac{1}{\omega_{n}}\mathcal{H}^{n},0_{n})$.
Thus $\mathcal{S}'\times\{1\}\subset C(X)\setminus\tilde{\mathcal{R}}_{n}$, and by the cone structure,
$(\mathcal{S}'\times[0,\infty)/\mathcal{S}'\times\{0\})\subset C(X)\setminus\tilde{\mathcal{R}}_{n}$.

Suppose $m_{X}(\mathcal{S}')=\frac{1}{\alpha}\mathcal{H}^{n-1}(\mathcal{S}')>0$, then from the co-area formula (\ref{conic3}), we have
$\mathcal{H}^{n}(\mathcal{S}'\times[1,2])>0$.
This contradicts to the fact that $\mathcal{H}^{n}(C(X)\setminus\tilde{\mathcal{R}}_{n})=0$.

Thus $m_{X}(X\setminus\mathcal{R}_{n-1})=0$.
This shows $\mathrm{dim}_{d_{X},m_{X}}X=n-1$.
By Theorem 4.1 of \cite{AHT17}, it holds $m_{X}(\mathcal{R}_{n-1}\setminus \mathcal{R}^{*}_{n-1})=0$.
Now we have $m_{X}(X\setminus \mathcal{R}^{*}_{n-1})=0$, hence  $\mathcal{H}^{n-1}(X)=\mathcal{H}^{n-1}(\mathcal{R}^{*}_{n-1})$.

Finally, by (\ref{1.11116}) and (\ref{1.11115}), we obtain (\ref{1.11117}).
The proof is completed.
\end{proof}

\begin{rem}\label{rem2.7}
It is well known that (\ref{1.11117}) is equivalent to
\begin{align}\label{1.11114}
\lim_{i\rightarrow\infty}\frac{\lambda_{i}}{i^{\frac{2}{n-1}}}=\frac{(2\pi)^{2}}{(n\omega_{n-1}\alpha)^{\frac{2}{n-1}}},
\end{align}
where $\lambda_{i}$ is the $i$-th Neumann eigenvalue (counted with multiplicity) on $(X, d_{X}, m_{X})$.
\end{rem}

\section{Harmonic functions on cones}\label{s3}

We study harmonic functions on cones in this section.
Even though what we need in this paper are results on cone-liked Ricci-limit spaces, we obtain results in the general setting of $\mathrm{RCD}$ spaces here.
More precisely, in this section we consider the following setting.
Suppose $(X, d_{X}, m_{X})$ is a $\textmd{RCD}^{*}(n-2,n-1)$ space, $(C(X),d_{C(X)},m_{C(X)})$ is a $(0, n-1)$-cone over $(X, d_{X}, m_{X})$, hence a $\textmd{RCD}^{*}(0,n)$ space.
Denote by $p_{\infty}$ the vertex of $C(X)$, and by $B_{p_{\infty}}(R)$ the open geodesic balls of radius $R$ centred at $p_{\infty}$.
Let $\Delta_{X}$ be the Laplacian operator on $(X,d_{X},m_{X})$.
Let $0=\lambda_{0}<\lambda_{1}\leq\lambda_{2}\leq\ldots$  be the Neumann eigenvalues of $\Delta_{X}$, $\{\varphi_{i}(x)\}_{i=0}^{\infty}$ be the corresponding eigenfunctions, i.e.
\begin{align}\label{2.0000}
-\Delta_{X}\varphi_{i}(x)=\lambda_{i}\varphi_{i}(x).
\end{align}
Furthermore, we require that $\{\varphi_{i}(x)\}_{i=0}^{\infty}$ forms a $L^{2}(X,m_{X})$-orthonormal basis and each $\varphi_{i}$ is Lipschitz.

For calculus in general metric spaces, the readers can refer to \cite{AGS14-1} \cite{C99} \cite{Gig15} etc.

We study harmonic functions on $(C(X),d_{C(X)},m_{C(X)})$ in this section.
A function $u$ defined on an open set $U\subset C(X)$ is called harmonic if $u\in W^{1,2}_{\mathrm{loc}}(U)$ and
$$\int_{U}\langle Du, Dv\rangle=0$$
for any Lipschitz function $v$ with $\mathrm{supp}(v)\subset\subset U$.
Here $\langle D\cdot, D\cdot\rangle:W^{1,2}_{\mathrm{loc}}(U)\times W^{1,2}_{\mathrm{loc}}(U)\rightarrow L^{1}_{\mathrm{loc}}(U)$ means the symmetric bilinear map defined by $\langle Df, Dg\rangle:=\lim_{\epsilon\downarrow0}\frac{|D(f+\epsilon g)|^{2}-|Df|^{2}}{2\epsilon}$ for any $f,g\in W^{1,2}_{\mathrm{loc}}(U)$ (here $|Df|$ is the minimal weak upper gradient of $f$), see \cite{AGS14-1} \cite{Gig15}.

Given a function $f :C(X)\rightarrow\mathbb{R}$ and $x\in X $ we denote by $f^{(x)}:\mathbb{R}^{+}\rightarrow\mathbb{R}$ the function given by $f^{(x)}(r) := f(x,r)$.
Similarly, for $r\in \mathbb{R}^{+}$ we denote by $f^{(r)}:X\rightarrow\mathbb{R}$ the function given by $f^{(r)}(x):= f(x,r)$.
For any $f\in W^{1,2}(C(X))$, we have
\begin{description}
  \item[(i)] $f^{(x)}\in W^{1,2}(\mathbb{R}^{+},r^{n-1}\mathcal{L}^{1})$ holds for $m_{X}$-a.e. $x\in X$;
  \item[(ii)] $f^{(r)}\in W^{1,2}(X,m_{X})$ holds for $r^{n-1}\mathcal{L}^{1}$-a.e. $r\in \mathbb{R}^{+}$;
  \item[(iii)] $|Df|^{2}(x,r)=\frac{1}{r^{2}}|D_{X}f^{(r)}|^{2}(x)+|D_{\mathbb{R}}f^{(x)}|^{2}(r)$ for $m_{C(X)}$-a.e. $(x,r)$.
\end{description}
Similar results holds for more general warped products spaces, and the assumption that $X$ is a $\mathrm{RCD}$ space is in fact redundant, see \cite{GH15}.

In the case that $u$ is a harmonic function defined on open set $U\subset C(X)$, from \cite{J14}, $|Du|\in L^{\infty}_{\mathrm{loc}}(U)$, and thus $u$ always has a locally Lipschitz representative.
In the remaining part of this paper, when we talk about a harmonic function $u$, we always mean the locally Lipschitz representative of it.
From Theorem \ref{eigenfcn} and Proposition \ref{prop3.2} below, we know that along the $r$ direction, $|Du|$ has better regularity.

In the remaining part of this section, let's fix a $R>0$ ($R$ may be $+\infty$).
Without loss of generality, we assume $R>1$.

\begin{thm}\label{eigenfcn}
If $u$ is a harmonic function defined on $B_{p_{\infty}}(R)\subset C(X)$ with $u(p_{\infty})=0$,
then
\begin{align}\label{har_fcn}
u(x,r)=\sum_{i=1}^{\infty}c_{i}r^{\alpha_{i}}\varphi_{i}(x),
\end{align}
where the convergence in (\ref{har_fcn}) is pointwise uniformly on any compact set of $B_{p_{\infty}}(R)$ and also in $W^{1,2}_{\mathrm{loc}}(B_{p_{\infty}}(R))$ sense.
In (\ref{har_fcn}), $\lambda_{i}$ and $\varphi_{i}$ are defined in (\ref{2.0000}), and $\lambda_{i}=\alpha_{i}(n+\alpha_{i}-2)$, $\alpha_{i}>0$; $c_{i}$ are constants and satisfy
\begin{align}\label{3.33336}
\sum_{i=1}^{\infty}\alpha_{i}c_{i}^{2}r^{2\alpha_{i}}<\infty
\end{align}
for any $r\in(0,R)$.
In particular, $c_{i}r^{\alpha_{i}}\rightarrow0$ as $i\rightarrow\infty$.
\end{thm}

We remark that (\ref{har_fcn}) is known to experts, and similar results have been appeared in many papers, see \cite{CM97b}, \cite{H15}, \cite{X15}, \cite{H16} etc.
We provide a detailed proof here, one reason is that there is no a proof on general $\mathrm{RCD}$ space in other reference, another reason is that some facts in the proof is useful in our proof of Proposition \ref{prop3.2}.

\begin{proof}[Proof of Theorem \ref{eigenfcn}]
Firstly, we can prove that for $\alpha_{i}$, $\varphi_{i}$ ($i\geq0$) with $\lambda_{i}=\alpha_{i}(\alpha_{i}+n-2)$ and $-\Delta_{X}\varphi_{i}(x)=\lambda_{i}\varphi_{i}(x)$, a function of the form $u(x,r)=r^{\alpha_{i}}\varphi_{i}(x)$ is harmonic on $C(X)$.
This can be easily checked by the definitions and the cone structure of $C(X)$.
We omit the details, see (\ref{3.33333}) for similar calculations.

For any $s\in(0,R)$, let $\bar{H}^{(s)}$ denote the space of Lipschitz functions $v$ defined on $\overline{B_{p_{\infty}}(s)}$ such that $v(p_{\infty})=0$, $v|_{B_{p_{\infty}}(s)}\in W^{1,2}(B_{p_{\infty}}(s))$ is harmonic.
Let $v_{k}(x,r)=r^{\alpha_{k}}\varphi_{k}(x)$, then for any $k\geq1$, $v_{k}$ belongs to $\bar{H}^{(s)}$.

It is easy to see that the bilinear form $\tilde{E}_{s}$ defined by
\begin{align}
\tilde{E}_{s}(u,v)=\int_{B_{p_{\infty}}(s)}\langle Du,Dv\rangle dm_{C(X)}\nonumber
\end{align}
is an inner product on $\bar{H}^{(s)}$.

For any $v\in\bar{H}^{(s)}$ and $v_{k}$ ($k\geq1$), we have
\begin{align}\label{3.33333}
&\tilde{E}_{s}(v_{k},v)\\
=&\int_{B_{p_{\infty}}(s)}\langle D(r^{\alpha_{k}}\varphi_{k}(x)),Dv(x,r)\rangle dm_{C(X)}\nonumber\\
=&\int_{B_{p_{\infty}}(s)}\bigl(\langle \alpha_{k}r^{\alpha_{k}-1}\varphi_{k}(x)D_{\mathbb{R}}r, D_{\mathbb{R}}v(x,r)\rangle+\frac{1}{r^{2}}\langle r^{\alpha_{k}}D_{X}\varphi_{k}(x),D_{X}v(x,r)\rangle \bigr) dm_{C(X)}\nonumber\\
=&\int_{0}^{s}r^{n-1}dr\int_{X}\bigl(\langle \alpha_{k}r^{\alpha_{k}-1}\varphi_{k}(x)D_{\mathbb{R}}r, D_{\mathbb{R}}v(x,r) \rangle+\frac{1}{r^{2}}\langle r^{\alpha_{k}}D_{X}\varphi_{k}(x),D_{X}v(x,r)\rangle \bigr)dm_{X}\nonumber\\
=&\int_{0}^{s}\alpha_{k}r^{\alpha_{k}+n-2}dr\int_{X}\varphi_{k}(x)\frac{\partial}{\partial r}v(x,r) dm_{X}+\int_{0}^{s}r^{n+\alpha_{k}-3}dr\int_{X}\langle D_{X}\varphi_{k}(x),D_{X}v(x,r)\rangle dm_{X}\nonumber\\
=&\int_{0}^{s}\alpha_{k}r^{\alpha_{k}+n-2}dr\frac{d}{d r}\biggl(\int_{X}\varphi_{k}(x)v(x,r) dm_{X}\biggr)+\int_{0}^{s}r^{n+\alpha_{k}-3}dr\int_{X}\lambda_{k}\varphi_{k}(x)v(x,r) dm_{X}\nonumber\\
=&\biggl(\alpha_{k}r^{\alpha_{k}+n-2}\int_{X}\varphi_{k}(x)v(x,r) dm_{X}\biggr)\biggl|_{r=0}^{r=s}-\int_{0}^{s}\alpha_{k}(\alpha_{k}+n-2)r^{\alpha_{k}+n-3}dr\int_{X}\varphi_{k}(x)v(x,r) dm_{X}\nonumber\\
&+\int_{0}^{s}r^{n+\alpha_{k}-3}dr\int_{X}\lambda_{k}\varphi_{k}(x)v(x,r) dm_{X}\nonumber\\
=&\alpha_{k}s^{\alpha_{k}+n-2}\int_{X}\varphi_{k}(x)v(x,s) dm_{X},\nonumber
\end{align}
where in the last equality we use $n\geq2$ and $\alpha_{k}>0$.

In particular, for $v_{i}$, $v_{j}$ with $i,j\geq1$, we have
\begin{align}
\tilde{E}_{s}(v_{i},v_{j})=\alpha_{i}s^{2\alpha_{i}+n-2}\delta_{ij},
\end{align}
this means that $\{\frac{1}{\alpha_{i}^{\frac{1}{2}}s^{\alpha_{i}+\frac{n}{2}-1}}v_{i}\}_{i\geq1}$ is an orthonormal sequence with respect to the inner product $\tilde{E}_{s}$.

Thus by Bessel's inequality, for any $v\in\bar{H}^{(s)}$, we have
\begin{align}\label{3.33331}
\int_{B_{p_{\infty}}(s)}|Dv|^{2} dm_{C(X)}
&\geq\sum_{i=1}^{\infty}\biggl(\frac{1}{\alpha_{i}^{\frac{1}{2}}s^{\alpha_{i}+\frac{n}{2}-1}} \tilde{E}_{s}(v_{i},v) \biggr)^{2}\\
&=\sum_{i=1}^{\infty}\alpha_{i}s^{n-2}\biggl(\int_{X}\varphi_{i}(x)v(x,s) dm_{X}\biggr)^{2}.\nonumber
\end{align}

Now suppose $u$ is a harmonic function defined on $B_{p_{\infty}}(R)\subset C(X)$ with $u(p_{\infty})=0$.
For any $s\in(0,R)$, we consider $u^{(s)}(\cdot)=u(\cdot,s)$.
By the spectral theorem, $u^{(s)}(x)=\sum_{i=0}^{\infty} \tilde{c}_{i}^{(s)}\varphi_{i}(x)$ holds in $L^{2}(X,m_{X})$-sense, where
$\tilde{c}_{i}^{(s)}=\int_{X} u^{(s)}(x)\varphi_{i}(x)dm_{X}$
for every $i\geq0$.
By Parseval's identity, we have
\begin{align}\label{3.33332}
\sum_{i=0}^{\infty}(\tilde{c}_{i}^{(s)})^{2}=|u^{(s)}|^{2}_{L^{2}(X)}<\infty.
\end{align}

Since the restriction of $u$ on $\overline{B_{p_{\infty}}(s)}$ belongs to $\bar{H}^{(s)}$, (\ref{3.33331}) means
\begin{align}\label{3.33334}
\sum_{i=1}^{\infty}\alpha_{i}(\tilde{c}_{i}^{(s)})^{2}s^{n-2}\leq \int_{B_{p_{\infty}}(s)}|Du|^{2} dm_{C(X)}<\infty.
\end{align}

For any $k\in\mathbb{Z}^{+}$, let
$u_{k}(x,r)=\sum_{i=0}^{k}\tilde{c}_{i}^{(s)}(\frac{1}{s})^{\alpha_{i}}r^{\alpha_{i}}\varphi_{i}(x)$.
Then for any pair of integers $k<l$, it is easy to check
\begin{align}\label{3.441}
\int_{B_{p_{\infty}}(s)}|D(u_{k}-u_{l})|^{2}dm_{C(X)}=\sum_{i=k+1}^{l}(\tilde{c}_{i}^{(s)})^{2}\alpha_{i}s^{n-2},
\end{align}
and
\begin{align}\label{3.442}
\int_{B_{p_{\infty}}(s)}|u_{k}-u_{l}|^{2}dm_{C(X)}
=&\int_{B_{p_{\infty}}(s)}|\sum_{i=k+1}^{l}\tilde{c}_{i}^{(s)}(\frac{r}{s})^{\alpha_{i}}\varphi_{i}(x)|^{2} dm_{C(X)}\\
=&\int_{0}^{s}r^{n-1}dr\sum_{i=k+1}^{l}(\tilde{c}_{i}^{(s)})^{2}\biggl(\frac{r}{s}\biggr)^{2\alpha_{i}}\nonumber\\
\leq&\sum_{i=k+1}^{l}(\tilde{c}_{i}^{(s)})^{2}\frac{1}{n}s^{n}.\nonumber
\end{align}

From  (\ref{3.33332}), (\ref{3.33334}), (\ref{3.441}) and (\ref{3.442}), it is easy to see that $\{u_{k}\}$ is a Cauchy sequence with respect to the $W^{1,2}(B_{p_{\infty}}(s))$ norm, thus there is a limit $\tilde{u}\in W^{1,2}(B_{p_{\infty}}(s))$.

Then one can easily prove that $\tilde{u}$ is harmonic in $B_{p_{\infty}}(s)$.

Furthermore, since $|u_{k}|_{L^{2}(B_{p_{\infty}}(s))}$ has a upper bound independent of $k$, by Theorem 1.1 of \cite{J14}, $|Du_{k}|_{L^{\infty}(B_{p_{\infty}}(r))}$ has a upper bound independent of $k$ for every $r\in(0,s)$.
Note that all $u_{k}$ are continuous functions, by Arzel\`{a}-Ascoli Theorem, $u_{k}$ converge pointwise uniformly on any compact subset of $B_{p_{\infty}}(s)$ to a continuous function $\bar{u}(x,r)=\sum_{i=0}^{\infty}\tilde{c}_{i}^{(s)}(\frac{1}{s})^{\alpha_{i}}r^{\alpha_{i}}\varphi_{i}(x)$, which is a locally Lipschitz representative of $\tilde{u}$.

By an argument similar to P. 273-P. 275 in \cite{EV10}, one can prove that $u-\bar{u}\in W^{1,2}_{0}(B_{p_{\infty}}(s))$.
Since $u-\bar{u}$ is a harmonic function on $B_{p_{\infty}}(s)$, by the maximum principle (see Theorem 7.17 in \cite{C99}), we know $u$ and $\bar{u}$ coincide on $B_{p_{\infty}}(s)$.
Since $\bar{u}(p)=u(p)=0$, we have $\tilde{c}_{0}^{(s)}=0$.

Recall that in the spectral theorem, the Fourier coefficients of a $L^{2}$-function with respect to the base $\{\varphi_{i}\}$ are uniquely determined.
Hence one can easily prove that $\tilde{c}_{i}^{(s)}(\frac{1}{s})^{\alpha_{i}}$ is independent of the choice of $s\in(0,R)$ (we omit the details here), and we can denote $c_{i}=\tilde{c}^{(s)}_{i}(\frac{1}{s})^{\alpha_{i}}$.
Finally, (\ref{3.33336}) follows by (\ref{3.33334}).
\end{proof}

Let $\bar{H}\subset W^{1,2}(B_{p_{\infty}}(R))\cap C(B_{p_{\infty}}(R))$ be the space of harmonic functions $u$ such that $u(p_{\infty})=0$.
A function of the form $u(x,r)=r^{\alpha_{i}}\varphi_{i}(x)$ (for $i\geq1$) belongs to $\bar{H}$.
From Theorem \ref{eigenfcn}, it is easy to see that, for any $r\in(0,R)$, the bilinear form $\tilde{E}_{r}$ defined by
\begin{align}\label{3.111}
\tilde{E}_{r}(u,v)=\int_{B_{p_{\infty}}(r)}\langle Du,Dv\rangle dm_{C(X)}
\end{align}
is an inner product on $\bar{H}$.

For any $u, v\in\bar{H}$, since $\tilde{E}_{r}(u,v)$ is absolutely continuous with respect to $r$,
$\frac{d}{ds}\bigl|_{s=r}\tilde{E}_{s}(u,v)$ is well-defined for a.e. $r\in(0,R)$.
In fact, from the cone structure and Theorem \ref{eigenfcn}, we have the following stronger result:

\begin{prop}\label{prop3.2}
For any $u, v\in\bar{H}$, suppose $u(x,r)=\sum_{i=1}^{\infty}c_{i}r^{\alpha_{i}}\varphi_{i}(x)$ and
$v(x,r)=\sum_{i=1}^{\infty}\bar{c}_{i}r^{\alpha_{i}}\varphi_{i}(x)$ as in Theorem \ref{eigenfcn}, then
\begin{align}\label{3.11112}
\tilde{E}_{r}(u,v)=\sum_{i=1}^{\infty}c_{i}\bar{c}_{i}\alpha_{i}r^{2\alpha_{i}+n-2}.
\end{align}
Furthermore, $\tilde{E}_{r}(u,v)$ is a $C^{1}$-function with respect to $r\in(0,R)$, and
\begin{align}\label{3.11113}
\frac{d}{ds}\biggl|_{s=r}\tilde{E}_{s}(u,v)=\sum_{i=1}^{\infty}c_{i}\bar{c}_{i}(\alpha_{i}^{2}+\lambda_{i})r^{2\alpha_{i}+n-3}
\end{align}
for any $r\in(0,R)$.
\end{prop}

\begin{rem}
From the proof of Proposition \ref{prop3.2}, we obtain that, for any $u\in\bar{H}$ such that $u(x,r)=\sum_{i=1}^{\infty}c_{i}r^{\alpha_{i}}\varphi_{i}(x)$ as in Theorem \ref{eigenfcn}, then
\begin{align}
\sum_{i=1}^{\infty}c_{i}^{2}(\alpha_{i}^{2}+\lambda_{i})r^{2\alpha_{i}}<\infty
\end{align}
holds for any $r\in(0,R)$.
\end{rem}

\begin{proof}[Proof of Proposition \ref{prop3.2}]
We prove the case $u=v\in\bar{H}$.

Firstly, suppose $u$ is a function of the form $u(x,r)=\sum_{i=1}^{k}c_{i}r^{\alpha_{i}}\varphi_{i}(x)$, then from the calculations similar to the proof of Theorem \ref{eigenfcn}, we have
\begin{align}
\int_{B_{p_{\infty}}(s)}|Du|^{2}dm_{C(X)}&=\int_{0}^{s}\bigl[\sum_{i=1}^{k}c_{i}^{2}(\alpha_{i}^{2}+\lambda_{i})r^{2\alpha_{i}+n-3} \bigr]dr\\
&=\sum_{i=1}^{k}c_{i}^{2}\alpha_{i}s^{2\alpha_{i}+n-2},\nonumber
\end{align}
and
\begin{align}\label{3.22221}
\frac{d}{ds}\biggl|_{s=r}\tilde{E}_{s}(u,u)=\sum_{i=1}^{k}c_{i}^{2}(\alpha_{i}^{2}+\lambda_{i})r^{2\alpha_{i}+n-3}.
\end{align}

Now we consider the case of general $u\in\bar{H}$.
Let $u(x,r)=\sum_{i=1}^{\infty}c_{i}r^{\alpha_{i}}\varphi_{i}(x)$ as in Theorem \ref{eigenfcn}.
For any integer $k$, let $u_{k}:B_{p_{\infty}}(R)\rightarrow\mathbb{R}$ be functions defined by $u_{k}(x,r)=\sum_{i=1}^{k}c_{i}r^{\alpha_{i}}\varphi_{i}(x)$;
let $a_{k}:(0,R)\rightarrow\mathbb{R}$ be functions defined by $a_{k}(r)=\sum_{i=1}^{k}c_{i}^{2}(\alpha_{i}^{2}+\lambda_{i})r^{2\alpha_{i}+n-3}$.

For any interval $I=(0,t_{1}]\subset(0,R)$, we fix a $t<R$ such that $t>t_{1}$.

Suppose $r\in I$.
For any pair of sufficiently large positive integers $k,l$ with $k<l$, we have
\begin{align}\label{3.11111}
|a_{k}(r)-a_{l}(r)|&=\sum_{i=k+1}^{l}c_{i}^{2}(\alpha_{i}^{2}+\lambda_{i})r^{2\alpha_{i}+n-3}\\
&=r^{n-3}\sum_{i=k+1}^{l}\bigl(c_{i}t^{\alpha_{i}}\bigr)^{2}(\alpha_{i}^{2}+\lambda_{i})\biggl(\frac{r}{t}\biggr)^{2\alpha_{i}} \nonumber\\
&\leq r^{n-3}\sum_{d=\lfloor\alpha_{k+1}\rfloor}^{\infty}\sum_{d-1<\alpha_{i}\leq d} \bigl(c_{i}t^{\alpha_{i}}\bigr)^{2}(\alpha_{i}^{2}+\lambda_{i})\biggl(\frac{r}{t}\biggr)^{2\alpha_{i}} \nonumber\\
&\overset{(1)}\leq Cr^{n-3}\sum_{d=\lfloor\alpha_{k+1}\rfloor}^{\infty}\biggl[d^{2}\sum_{d-1<\alpha_{i}\leq d} \biggl(\frac{r}{t}\biggr)^{2d-2}\biggr] \nonumber\\
&\overset{(2)}\leq Cr^{n-3}\sum_{d=\lfloor\alpha_{k+1}\rfloor}^{\infty}d^{n+1} \biggl(\frac{r}{t}\biggr)^{2d-2}, \nonumber
\end{align}
where in the equality (1), we use the fact that $\alpha_{i}^{2}+\lambda_{i}=\alpha_{i}(2\alpha_{i}+n-2)\leq 3d^{2}$ for $d-1<\alpha_{i}\leq d$ and $i$ large and the fact that $c_{i}t^{\alpha_{i}}\rightarrow0$ as $i\rightarrow\infty$;
while in the inequality (2), we use the fact that the number of $i$'s such that $\alpha_{i}\leq d$ is bounded from above by $Cd^{n-1}$ for some constant $C=C(n)$.
The later fact can be proved as follows: since $\{v_{i}(x,r)=r^{\alpha_{i}}\varphi_{i}(x)\}_{i\geq1}$ are linear independent harmonic functions on $C(X)$, the number of $i$ such that $\alpha_{i}\leq d$ is bounded from above by $\textmd{dim}(\mathcal{H}_{d}(C(X)))$, while $\textmd{dim}(\mathcal{H}_{d}(C(X)))$ is bounded from above by $Cd^{n-1}$ (see e.g. \cite{CM98b}, \cite{L97}, \cite{HKX13}).

Note that $$\sum_{d=\lfloor\alpha_{k+1}\rfloor}^{\infty}d^{n+1} \biggl(\frac{r}{t}\biggr)^{2d-2}\rightarrow0$$
uniformly for $r\in I$ as $k\rightarrow\infty$.
Thus as $k\rightarrow\infty$, $a_{k}(r)$ converges uniformly on $I$ to some continuous function $a(r)=\sum_{i=1}^{\infty}c_{i}^{2}(\alpha_{i}^{2}+\lambda_{i})r^{2\alpha_{i}+n-3}$.

By (\ref{3.22221}), we have $\frac{d}{ds}\bigl|_{s=r}\tilde{E}_{s}(u_{k},u_{k})=a_{k}(r)$.
In addition, by Theorem \ref{eigenfcn}, $\tilde{E}_{s}(u_{k},u_{k})\rightarrow \tilde{E}_{s}(u,u)$ holds for any $s\in I$.
Thus we have
\begin{align}\label{3.22222}
\frac{d}{ds}\biggl|_{s=r}\tilde{E}_{s}(u,u)=a(r).
\end{align}
By the arbitrariness of $I$, (\ref{3.22222}) holds for any $r\in(0,R)$.

Finally, the general case when $u, v\in\bar{H}$ can easily done by polarization.
The proof is completed.
\end{proof}

\begin{rem}
For $u\in\bar{H}$ such that $u(x,r)=\sum_{i=1}^{\infty}c_{i}r^{\alpha_{i}}\varphi_{i}(x)$ as in Theorem \ref{eigenfcn}, then similar to the arguments as above, for any $r\in(0,R)$, we have
\begin{align}\label{3.22227}
\int_{X}u^{(r)}(x)^{2}dm_{X}=\sum_{i=1}^{\infty}c_{i}^{2}r^{2\alpha_{i}},
\end{align}
\begin{align}\label{3.22226}
\int_{X}|D_{X}u^{(r)}(x)|^{2}dm_{X}=\sum_{i=1}^{\infty}c_{i}^{2}r^{2\alpha_{i}}\lambda_{i}.
\end{align}
\end{rem}

The proof of the following lemma can be found in \cite{LW99}:

\begin{lem}[Lemma 1.2 in \cite{LW99}]\label{lem3.1}
Let $V$ be a $k$-dimensional subspace of a vector space $W$. Assume that $W$ is endowed with an inner product $L$ and a bilinear form $\Phi$. Then for any given linearly independent set of vectors $\{w_{1},\ldots,w_{k-1}\}\subset W$, there exists an orthonormal basis $\{v_{1},\ldots,v_{k}\}$ of $V$ with respect to $L$ such that $\Phi(v_{i},w_{j}) = 0$ for all $1\leq j < i\leq k$.
\end{lem}
Let $\tilde{H}$ be a $k$-dimensional subspace of $\bar{H}$.
By Lemma \ref{lem3.1}, for any $r\in(0,R)$, we can always find $\{v_{1},\ldots,v_{k}\}$ such that they form an $\tilde{E}_{r}$-orthonormal base of $\tilde{H}$ and
$$\int_{X}v_{i}^{(r)}(x)\varphi_{j}(x)dm_{X}=0$$
for all $1\leq j < i\leq k$.
Then by the min-max principle, we have
\begin{align}\label{3.22223}
\lambda_{i}\int_{X}v_{i}^{(r)}(x)^{2}dm_{X}\leq\int_{X}|D_{X}v_{i}^{(r)}(x)|^{2}dm_{X}
\end{align}
for $1\leq i\leq k$.
For any $1\leq i\leq k$, suppose $v_{i}(x,r)=\sum_{j=1}^{\infty}c_{i,j}r^{\alpha_{j}}\varphi_{j}(x)$ as in Theorem \ref{eigenfcn}, then by (\ref{3.22227}), (\ref{3.22226}) and (\ref{3.22223}), we have
\begin{align}\label{3.11114}
\lambda_{i}\sum_{j=1}^{\infty}c_{i,j}^{2}r^{2\alpha_{j}}\leq \sum_{j=1}^{\infty}c_{i,j}^{2}r^{2\alpha_{j}}\lambda_{j}.
\end{align}

For $1\leq i\leq k$, by (\ref{3.11112}), (\ref{3.11114}) and the Cauchy-Schwarz inequality, we have
\begin{align}\label{3.22224}
\lambda_{i}^{\frac{1}{2}}&=\lambda_{i}^{\frac{1}{2}}\tilde{E}_{r}(v_{i},v_{i}) =\lambda_{i}^{\frac{1}{2}}r^{n-2}\sum_{j=1}^{\infty}c_{i,j}^{2}\alpha_{j}r^{2\alpha_{j}}\\
&\leq \frac{r^{n-2}}{2}\biggl[\lambda_{i}\sum_{j=1}^{\infty}c_{i,j}^{2}r^{2\alpha_{j}}+
\sum_{j=1}^{\infty}c_{i,j}^{2}\alpha_{j}^{2}r^{2\alpha_{j}}\biggr]\nonumber\\
&\leq \frac{r^{n-2}}{2}\biggl[\sum_{j=1}^{\infty}c_{i,j}^{2}\lambda_{j}r^{2\alpha_{j}}+
\sum_{j=1}^{\infty}c_{i,j}^{2}\alpha_{j}^{2}r^{2\alpha_{j}}\biggr]\nonumber\\
&=\frac{1}{2}\sum_{j=1}^{\infty}c_{i,j}^{2}r^{2\alpha_{j}+n-2}(\lambda_{j}+\alpha_{j}^{2}).\nonumber
\end{align}
Summing up (\ref{3.22224}) for $1\leq i\leq k$, and by Proposition \ref{prop3.2}, we have
\begin{align}\label{3.22225}
\sum_{i=1}^{k}\lambda_{i}^{\frac{1}{2}}\leq\frac{r}{2}\sum_{i=1}^{k}\frac{d}{ds}\biggl|_{s=r}\tilde{E}_{s}(v_{i},v_{i}).
\end{align}

For $s,t\in(0,R)$, let $\mathrm{det}_{s}\tilde{E}_{t}$ be the determinant of ${\tilde{E}}_{t}$ with respect to ${\tilde{E}}_{s}$ on $\tilde{H}$.

\begin{lem}\label{lem3.4}
The function $s\mapsto\ln\mathrm{det}_{1}\tilde{E}_{s}$ is $C^{1}$ in $(0,R)$.
Furthermore, suppose $\{v_{1},\ldots,v_{k}\}$ is an $\tilde{E}_{r}$-orthonormal base of $\tilde{H}$, then
\begin{align}\label{3.23331}
\frac{d}{ds}\biggl{|}_{s=r} \ln\mathrm{det}_{1}\tilde{E}_{s} =\sum_{i=1}^{k}\frac{d}{ds}\biggl{|}_{s=r}\tilde{E}_{s}(v_{i},v_{i}).
\end{align}
\end{lem}

\begin{proof}
The conclusion that the function $s\mapsto\ln\mathrm{det}_{1}\tilde{E}_{s}$ is $C^{1}$ is obvious from the definitions and Proposition \ref{prop3.2}.
Thus we only need to prove (\ref{3.23331}).

From Theorem \ref{eigenfcn} and Proposition \ref{prop3.2}, it is easy to see that the right hand side of (\ref{3.23331}) is independent of the choice of the $\tilde{E}_{r}$-orthonormal bases of $\tilde{H}$.

Suppose $\{v_{1},\ldots,v_{k}\}$ is an $\tilde{E}_{r}$-orthonormal base of $\tilde{H}$ that diagonalizes $\tilde{E}_{1}$.
Suppose $\tilde{E}_{1}(v_{i},v_{j})=\delta_{ij}\mu_{i}$ for $\mu_{i}>0$.
For any $s\in(0,R)$, let $\{h^{(s)}_{ij}\}_{i,j}$ be the matrix given by $h^{(s)}_{ij}=\tilde{E}_{s}(v_{i},v_{j})$, then
$$\ln\mathrm{det}_{1}\tilde{E}_{s}=\ln\mathrm{det}(h^{(s)}_{ij})-\ln\bigl(\prod_{i=1}^{k}\mu_{i}\bigr).$$
Let $c^{(s)}_{ij}$ be the cofactors of $h^{(s)}_{ij}$.
Note that $c^{(r)}_{ij}=\delta_{ij}$.
By Proposition \ref{prop3.2}, $h^{(s)}_{ij}$ and $c^{(s)}_{ij}$ are $C^{1}$-functions.
Hence
\begin{align}
\frac{d}{ds}\biggl{|}_{s=r} \ln\mathrm{det}_{1}\tilde{E}_{s}&= \frac{d}{ds}\biggl{|}_{s=r}\ln\mathrm{det}(h^{(s)}_{i,j}) =\frac{d}{ds}\biggl{|}_{s=r}\ln\sum_{j=1}^{k}h^{(s)}_{1j}c^{(s)}_{1j}\\
&=\frac{d}{ds}\biggl{|}_{s=r}h^{(s)}_{11}+\frac{d}{ds}\biggl{|}_{s=r}c^{(s)}_{11}.\nonumber
\end{align}
By induction, we have
\begin{align}
\frac{d}{ds}\biggl{|}_{s=r} \ln\mathrm{det}_{1}\tilde{E}_{s}= \sum_{i=1}^{k}\frac{d}{ds}\biggl{|}_{s=r}h^{(s)}_{ii},
\end{align}
i.e. (\ref{3.23331}) holds.
\end{proof}

By (\ref{3.22225}) and Lemma \ref{lem3.4}, we have

\begin{prop}\label{prop3.6}
Let $\tilde{H}$ be a $k$-dimensional subspace of $\bar{H}$, then
\begin{align}
\sum_{i=1}^{k}\lambda_{i}^{\frac{1}{2}}\leq\frac{r}{2}\frac{d}{ds}\biggl{|}_{s=r}\ln\mathrm{det}_{1}\tilde{E}_{s}
\end{align}
holds for every $r\in(0,R)$.
\end{prop}

\section{Taking limits}\label{s4}

From now on, let $\tau$ and $\delta$ be fixed sufficiently small positive numbers, and $A$ be a fixed large number.
At the end of Section \ref{s5}, $\tau$ and $\delta$ will eventually converge to $0$ while $A$ converging to $\infty$.

Suppose $(M^{n},g)$ is a complete manifold with nonnegative Ricci curvature and satisfies (\ref{1.01111}), where $\mu$ is the volume element determined by $g$.
Let $\rho$ be the distance determined by $g$, $p$ be a fixed point on $M$, and denote by $B_{p}(r)=\{x\in M\mid\rho(x,p)< r\}$.
We assume $(M^{n},g)$ has a unique tangent cone at infinity, which is denoted by $(M_{\infty}, p_{\infty}, \rho_{\infty},\nu_{\infty})=(C(X),d_{C(X)},m_{C(X)})$.
For any integer $d\geq1$, let
$$\mathcal{H}_{d}(M)=\{u\in C^{\infty}(M)\mid \Delta u=0, u(p)=0, |u(x)|\leq C(\rho(x)^{d}+1) \text{ for some }C\},$$
and $h_{d}:= \textmd{dim}\mathcal{H}_{d}(M)$.

Take $k_{0}=h_{0}=0$.
For any positive integer $i$, let $k_{i}=h_{i}-h_{i-1}$.

By Cheng-Yau's gradient estimate, for any $u\in \mathcal{H}_{d}(M)$, we always have
\begin{align}\label{4.11111}
|\nabla u|(x)\leq C(\rho(x)^{d-1}+1)
\end{align}
for some positive constant $C$ which depends on $u$.

For any $r>0$, let ${E}_{r}$ be the inner product of $\mathcal{H}_{d}(M)$ given by $${E}_{r}(u,v)=\int_{B_{p}(r)}\langle \nabla u, \nabla v \rangle d\mu.$$

For each positive integer $d$, by induction we can decompose the space $\mathcal{H}_{d}(M)$ with respect to the inner product ${E}_{1}$ into a direct sum
$$\mathcal{H}_{d}(M)=K^{1}\oplus\ldots\oplus K^{d},$$
so that each $K^{i}$ is a subspace of $\mathcal{H}_{d}(M)$ consisting of harmonic functions of growth order not grater than $i$ but larger than $i-1$.
Note that dimension of $K^{i}$ is $k_{i}$.

For $s,t>0$, let $\mathrm{det}_{s}{E}_{t}$ be the determinant of ${E}_{t}$ with respect to ${E}_{s}$.
Define $g(r):=\mathrm{det}_{1}{E}_{r}$.

Let $\{v_{1},\ldots,v_{h_{d}}\}$ be an $E_{1}$-orthonormal base of $\mathcal{H}_{d}(M)$ such that $\{v_{h_{i-1}+1},\ldots,v_{h_{i}}\}$ is an $E_{1}$-orthonormal base of $K^{i}$, then we have
\begin{align}\label{4.11011}
\textmd{det}_{1}{E}_{r}&=\textmd{det}({E}_{r}(v_{i},v_{j})) =\sum_{\sigma\in S_{h_{d}}}\biggl(\mathrm{sgn}(\sigma)\prod_{i=1}^{h_{d}}{E}_{r}(v_{i},v_{\sigma(i)}) \biggr) \\
&\leq C\prod_{i=1}^{h_{d}}\int_{B_{p}(r)}\langle Dv_{i},Dv_{i}\rangle d\mu
\leq C (1+r^{\sum_{i=1}^{d}(2i+n-2)k_{i}}),\nonumber
\end{align}
where $C$ are constants depending only on $\mathcal{H}_{d}(M)$ and $M$.
In (\ref{4.11011}), we use Cauchy-Schwarz inequality in the first inequality, and use (\ref{4.11111}) and the volume comparison theorem in the second inequality.
In conclusion, the function $g(r)$ is positive, nondecreasing and satisfies
$$g(r)\leq C(1+r^{s})$$
with
$$s=\sum_{i=1}^{d}(2i+n-2)k_{i}.$$

The proof of the following lemma can be found in \cite{CM97a}:
\begin{lem}[Lemma 3.1 in \cite{CM97a}]\label{4.11112}
Suppose $f_{1},\ldots,f_{l}$ are nonnegative nondecreasing functions defined on $(0,\infty)$ such that none of the $f_{i}$ vanishes identically, and there are $d,K>0$ such that $f_{i}(r)\leq K(r^{d}+1)$ for all $i$.
Then for every $\Omega>1$, $k\leq l$ and any $C>\Omega^{\frac{ld}{l-k+1}}$, there exist $k$ of these functions $f_{a_{1}},\ldots,f_{a_{k}}$ and infinitely many integers $m$ such that
$$f_{a_{i}}(\Omega^{m+1})\leq Cf_{a_{i}}(\Omega^{m})$$
for every $1\leq i\leq k$.
\end{lem}
In Lemma \ref{4.11112}, take $k=l=1$, $f_{1}=g$, by the above properties of $g$, if we take $\Omega=\beta:=1+\frac{1}{d}$, $C=\beta^{s+1}$, then we have
$g(\beta^{m_{i}+1})\leq \beta^{s+1}g(\beta^{m_{i}})$ for a sequence of positive numbers $\{m_{i}\}_{i\in\mathbb{N}^{+}}$ with $m_{i}\rightarrow\infty$.
For every $i$, let $R_{i}=\beta^{m_{i}}$, then we have
\begin{align}\label{5.1111}
g(\beta R_{i})\leq \beta^{s+1}g(R_{i}).
\end{align}

For fixed $i$, let $\{u_{1},\ldots,u_{h_{d}}\}$ be an ${E}_{R_{i}}$-orthonormal base of $\mathcal{H}_{d}(M)$ that diagonalizes ${E}_{\beta R_{i}}$, and let
\begin{align}\label{5.1112}
J:=\bigl\{1\leq j\leq h_{d} \bigl| {E}_{\beta R_{i}}(u_{j},u_{j})\leq\beta^{\frac{A(s+1)}{h_{d}}}\bigr\}.
\end{align}

Because
$$\beta^{s+1}\geq\prod_{j=1}^{h_{d}}{E}_{\beta R_{i}}(u_{j},u_{j})=\prod_{j\in J}{E}_{\beta R_{i}}(u_{j},u_{j})\cdot\prod_{j\notin J}{E}_{\beta R_{i}}(u_{j},u_{j})\geq\beta^{\frac{A(s+1)}{h_{d}}(h_{d}-\#J)},$$
we have
$$\# J\geq h_{d}\frac{A-1}{A}.$$

Take
\begin{align}\label{4.33331}
k=h_{d}\frac{A-1}{A},
\end{align}
and choose a subset $J'$ of $J$ such that $\# J'=k$, let
$$\bar{H}_{i}:=\textmd{span}\{u_{j}\}_{j\in J'}.$$
Without loss of generality, we assume $J'=\{1,2,\ldots,k\}$.

Let $u\in \bar{H}_{i}$ with $u=\Sigma_{j=1}^{k}b_{j}u_{j}$, then for any $r\in [R_{i},\beta R_{i}]$, we have
\begin{align}\label{4.00000}
&{E}_{r}(u,u)\geq {E}_{R_{i}}(u,u)=\sum_{j=1}^{k}b_{j}^{2}{E}_{R_{i}}(u_{j},u_{j})\\
\geq &\sum_{j=1}^{k}b_{j}^{2}\beta^{-\frac{A(s+1)}{h_{d}}}{E}_{\beta R_{i}}(u_{j},u_{j}) =\beta^{-\frac{A(s+1)}{h_{d}}}{E}_{\beta R_{i}}(u,u).\nonumber
\end{align}

Take
$r_{i}=(1+\frac{1-\tau}{d})R_{i}$, and $g_{i}=r_{i}^{-2}g$.
As $i\rightarrow\infty$, there is a subsequence of $\{r_{i}\}$ (still denoted by $\{r_{i}\}$) such that $(M_{i}, p, \rho_{i},\nu_{i})=(M,p,g_{i})$ converge to $(M_{\infty}, p_{\infty}, \rho_{\infty},\nu_{\infty})=(C(X),p_{\infty},d_{C(X)},m_{C(X)})$ in the pmGH sense.

%Denote by $D^{i}$ the Cheeger derivative with respect to $(M_{i},\rho_{i},\nu_{i})$.

Take
\begin{align}
t_{1}&=\frac{d}{d+1-\tau}, \\
t_{2}&=\frac{d+1}{d+1-\tau},\\
t_{3}&=\frac{1}{2}(t_{2}+1)=\frac{2d+2-\tau}{2d+2-2\tau}.
\end{align}

Denote by $B_{p}^{(i)}(r)=\{x\in M_{i}\mid\rho_{i}(x,p)< r\}$, $B_{p_{\infty}}(r)=\{x\in C(X)\mid d_{C(X)}(x,p_{\infty})< r\}$.
For any $r>0$, let ${E}^{(i)}_{r}$ be the inner product of $\mathcal{H}_{d}(M_{i})$ given by ${E}^{(i)}_{r}(u,v)=\int_{B^{(i)}_{p}(r)}\langle \nabla^{(i)} u, \nabla^{(i)} v \rangle d\nu_{i}$.

For any $u\in \bar{H}_{i}$, define $u^{(i)}:B_{p}^{(i)}(t_{2})\rightarrow\mathbb{R}$ by
\begin{align}\label{4.00001}
u^{(i)}(x)=\frac{(\mu(B_{p}(r_{i})))^{\frac{1}{2}}}{r_{i}({E}_{r_{i}}(u,u))^{\frac{1}{2}}}u(x).
\end{align}
Then $u^{(i)}$ is a harmonic function satisfying
\begin{align}\label{4.00002}
{E}^{(i)}_{1}(u^{(i)},u^{(i)})=1,
\end{align}
\begin{align}\label{4.00003}
{E}^{(i)}_{t_{2}}(u^{(i)},u^{(i)})\leq \beta^{\frac{A(s+1)}{h_{d}}}.
\end{align}

By (\ref{4.00003}) and Li-Schoen's mean-value inequality (see \cite{LS84}), we have
\begin{align}
\sup_{B^{(i)}_{p}(t_{3})}|\nabla^{(i)}u^{(i)}|\leq C(n,t_{2},\alpha)\beta^{\frac{A(s+1)}{2h_{d}}}.
\end{align}
Since $u^{(i)}(p)=0$, we have
\begin{align}
\sup_{B^{(i)}_{p}(t_{3})}|u^{(i)}|\leq C(n,t_{2},\alpha)\beta^{\frac{A(s+1)}{2h_{d}}}.
\end{align}

Note that the classical Arzel\`{a}-Ascoli Theorem can be generalized to functions living on different spaces, see e.g. Definition 2.11 and Proposition 2.12 in \cite{MN14}.
Thus up to a subsequence, $u^{(i)}$ converge pointwise and uniformly on any compact set $K\subset\subset B_{p_{\infty}}(t_{3})$ to some function $u^{(\infty)}$ defined on $B_{p_{\infty}}(t_{3})$.
Since every $u^{(i)}$ is a harmonic function, by  \cite{H11} (see also \cite{Xu14} \cite{D02}), $u^{(\infty)}$ is a harmonic function on $B_{p_{\infty}}(t_{3})$.

For harmonic functions defined on $B_{p_{\infty}}(t_{3})$, we use $\tilde{E}_{r}$ to denote the inner product defined as in (\ref{3.111}) (here we take $R=t_{3}$, $r\in(0,R)$).

Recall the following theorem:
\begin{thm}[\cite{H11}, see also \cite{D02} and \cite{Xu14}]\label{2.7777}
Suppose $(M_{i}, p, \rho_{i},\nu_{i})$ converge to $(M_{\infty}, p_{\infty}, \rho_{\infty},\nu_{\infty})$ in the pmGH sense.
Suppose $f_{i},g_{i}$ are harmonic functions defined on $B^{(i)}_{p}(R)\subset M_{i}$ for every $i$.
Suppose $f_{i}\rightarrow f_{\infty}$ and $g_{i}\rightarrow g_{\infty}$ uniformly on any compact subset $K\subset\subset B_{p_{\infty}}(R)\subset M_{\infty}$, where $f_{\infty}$, $g_{\infty}$ are harmonic functions defined on $B_{p_{\infty}}(R)$, then we have
$$\lim_{i\rightarrow\infty}\int_{B^{(i)}_{p}(r)}|\nabla^{(i)}f_{i}|^{2}d\nu_{i}
=\int_{B_{p_{\infty}}(r)}|Df_{\infty}|^{2}d\nu_{\infty},$$
$$\lim_{i\rightarrow\infty}\int_{B^{(i)}_{p}(r)}\langle\nabla^{(i)}f_{i},\nabla^{(i)}g_{i}\rangle d\nu_{i}
=\int_{B_{p_{\infty}}(r)}\langle Df_{\infty},Dg_{\infty}\rangle d\nu_{\infty}$$
for any $r\in(0,R)$.
\end{thm}

We remark that when the $(M_{i}, p, \rho_{i},\nu_{i})$ and $(M_{\infty}, p_{\infty}, \rho_{\infty},\nu_{\infty})$ are general $\mathrm{RCD}^{*}(K,N)$ spaces, there are also theory on convergence of functions defined on $(M_{i}, p, \rho_{i},\nu_{i})$.
Except for the notion of pointwise convergence, there are other notions of convergence such as $L^{2}$-weak convergence and $L^{2}$-strong convergence, see e.g. \cite{GMS15}.
For suitable family of functions, the above notions of convergence may be equivalent, see e.g. Proposition 3.2 in \cite{AHT17}.
Theorem \ref{2.7777} can be generalized to the $\mathrm{RCD}^{*}(K,N)$ setting, see e.g. Theorem 4.4 in \cite{AH17} (see also Corollary 3.3 in \cite{ZZ17}).

By (\ref{4.00002}), (\ref{4.00003}) and Theorem \ref{2.7777}, we have
\begin{align}\label{4.00004}
\tilde{E}_{1}(u^{(\infty)},u^{(\infty)})=1,
\end{align}
and
\begin{align}\label{4.00005}
\tilde{E}_{r}(u^{(\infty)},u^{(\infty)})\leq \beta^{\frac{A(s+1)}{h_{d}}}
\end{align}
for every $r\in(0,t_{3})$.

Suppose we have another function $v\in \bar{H}_{i}$, such that
$${E}_{r_{i}}(u,v)=0,$$
we define $v^{(i)}:B_{p}^{(i)}(t_{2})\rightarrow\mathbb{R}$ similar to (\ref{4.00001}), then up to a subsequence, $v^{(i)}$ converge to a harmonic function $v^{(\infty)}$ defined on $B_{p_{\infty}}(t_{3})$ satisfying properties similar to (\ref{4.00004}) (\ref{4.00005}).
Furthermore, by Theorem \ref{2.7777}, we have
\begin{align}
\tilde{E}_{1}(u^{(\infty)},v^{(\infty)})=0.
\end{align}
Thus $u^{(\infty)}$ and $v^{(\infty)}$ are orthonormal to each other with respect to the inner product $\tilde{E}_{1}$.

For every $i$, on $\bar{H}_{i}$ we choose an ${E}_{R_{i}}$-orthonormal base which also diagonalizes ${E}_{r_{i}}$, and then from the above process, we obtain $\tilde{E}_{1}$-orthonormal harmonic functions $\{u^{(\infty)}_{1},\ldots,u^{(\infty)}_{k}\}$ defined on $B_{p_{\infty}}(t_{3})$.
Furthermore, by Theorem \ref{2.7777}, we have
$$\tilde{E}_{t_{1}}(u^{(\infty)}_{i'},u^{(\infty)}_{j'})=0$$
for $i'\neq j'$.
%Assume $\tilde{E}_{t_{1}}(u^{(\infty)}_{m},u^{(\infty)}_{n})=\delta_{mn}\mu_{m}$ for $1\leq m,n\leq k$, where $\mu_{m}>0$.
Denoted by $\tilde{H}=\mathrm{span}\{u^{(\infty)}_{1},\ldots,u^{(\infty)}_{k}\}$.

For $s,t\in(0,t_{3})$, we consider the determinant $\mathrm{det}_{s}\tilde{E}_{t}$ on $\tilde{H}$ as in Section \ref{s3}.
Thus from the definition, Theorem \ref{2.7777} and (\ref{5.1111}), it is easy to see
\begin{align}\label{4.11117}
\mathrm{det}_{t_{1}}\tilde{E}_{1}&=\prod_{m=1}^{k} \frac{\tilde{E}_{1}(u^{(\infty)}_{m},u^{(\infty)}_{m})}{\tilde{E}_{t_{1}}(u^{(\infty)}_{m},u^{(\infty)}_{m})} =\lim_{i\rightarrow\infty}\mathrm{det}_{R_{i}}{E}_{r_{i}}\bigl|\bar{H}_{i}\\
&\leq\lim_{i\rightarrow\infty}\mathrm{det}_{R_{i}}{E}_{\beta R_{i}}\leq \beta^{s+1}.\nonumber
\end{align}

By Lemma \ref{lem3.4}, the function $s\mapsto\ln\mathrm{det}_{1}\tilde{E}_{s}$ is $C^{1}$.
By the mean value theorem there is an $r\in[t_{1},1]$ such that
\begin{align}\label{4.22221}
\frac{d}{ds}\biggl{|}_{s=r}\ln\mathrm{det}_{1}\tilde{E}_{s}
=&\frac{\ln\mathrm{det}_{1}\tilde{E}_{1}-\ln\mathrm{det}_{1}\tilde{E}_{t_{1}}}{1-t_{1}}\\
=&\frac{\ln\mathrm{det}_{t_{1}}\tilde{E}_{1}}{1-t_{1}} \leq\frac{(s+1)\ln \beta}{1-t_{1}},\nonumber
\end{align}
where we use (\ref{4.11117}) in the last inequality.

By Proposition \ref{prop3.6} and (\ref{4.22221}), we have
\begin{align}\label{4.99999}
\sum_{i=1}^{k}\lambda_{i}^{\frac{1}{2}}\leq \frac{r(s+1)\ln \beta}{2(1-t_{1})}\leq \frac{(s+1)\ln(1+\frac{1}{d})}{2(1-t_{1})}.
\end{align}

\section{Complete the proofs}\label{s5}

In \cite{D04}, it is proved that suppose $(M^{n},g)$ is a Riemannian manifold with nonnegative Ricci curvature and maximal volume growth, and suppose that $M$ has a unique tangent cone at infinity, then $h_{d}(M)\geq C d^{n-1}$ for some constant $C$ depending on $n$ and $\alpha$; see Theorem 0.1 of \cite{D04}.
During the proof of \cite{D04}, the author in fact obtain the following more stronger result (see also Theorem 1.2 of \cite{H16} for a more general result which may hold in some collapsed cases):
\begin{thm}[See \cite{D04} and \cite{H16}]\label{main-2}
Assume that $(M^{n},g)$ is a complete Riemannian manifold with nonnegative Ricci curvature and satisfies (\ref{1.01111}).
Assume that $M$ has a unique tangent cone at infinity, which is denoted by $(C(X),d_{C(X)},m_{C(X)})$.
Let $0=\lambda_{0}<\lambda_{1}\leq\lambda_{2}\leq\ldots$ be the Neumann eigenvalues on the cross section $(X, d_{X}, m_{X})$.
Let $N_{(X,d_{X},m_{X})}$ be the counting function.
Then given any $d>0$ such that $d(d+n-2)>\lambda_{1}$, we have
\begin{align}\label{5.33333}
h_{d}\geq N_{(X,d_{X},m_{X})}(d(d+n-2)-\epsilon)-1
\end{align}
for any $\epsilon>0$.
\end{thm}

We will complete the proof of Theorem \ref{main-4} in this section.
In the proof, we always assume $d$ is sufficiently large.
Since $k=h_{d}\frac{A-1}{A}$, $k$ is a large number.
Furthermore, note that
\begin{align}\label{5.00003}
s=\sum_{i=1}^{d}(2i+n-2)k_{i}\geq nh_{d}\geq Cd^{n-1},
\end{align}
hence $s$ is also a large number.

\begin{proof}[Proof of Theorem \ref{main-4}]
We first prove
\begin{align}\label{5.22225}
\limsup_{d\rightarrow\infty} d^{-n}\sum_{i=1}^{d}h_{i-1}\leq\frac{2\alpha}{n!\omega_{n}}.
\end{align}

For any integer $d$, taking limit (blow down the manifold) as in Section \ref{s4}.
Then we apply Proposition \ref{prop2.1111} to the tangent cone at infinity $(C(X),d_{C(X)},m_{C(X)},p_{\infty})$.
By Weyl's law we know (\ref{1.11114}) holds, thus for every small $\xi>0$, there exists $i_{0}$ (here we remark that $i_{0}$ depends on $X$ and $\xi$) such that for any $i\geq i_{0}$, we have
\begin{align}\label{5.00002}
\lambda_{i}^{\frac{1}{2}}\geq(1-\xi)C_{1}^{\frac{1}{n-1}}i^{\frac{1}{n-1}}(n\alpha)^{-\frac{1}{n-1}},
\end{align}
where
$$C_{1}:=\frac{(2\pi)^{n-1}}{\omega_{n-1}}=\frac{n!\omega_{n}}{2}.$$

On the other hand, it is easy to check
$$\sum_{i=i_{0}}^{k}i^{\frac{1}{n-1}}\geq\frac{n-1}{n}(k^{\frac{n}{n-1}}-i_{0}^{\frac{n}{n-1}}).$$
Hence
\begin{align}\label{5.22223}
\sum_{i=1}^{k}\lambda_{i}^{\frac{1}{2}} \geq(1-\xi)\frac{n-1}{n}C_{1}^{\frac{1}{n-1}}(n\alpha)^{-\frac{1}{n-1}}(k^{\frac{n}{n-1}}-i_{0}^{\frac{n}{n-1}}).
\end{align}

Since $1-t_{1}=\frac{1-\tau}{d+1-\tau}$ and $s\geq Cd^{n-1}$, for any given small $\delta>0$,
\begin{align}
\frac{(s+1)\ln(1+\frac{1}{d})}{2(1-t_{1})}\leq \frac{s}{2}(\frac{1}{1-\tau}+\delta)
\end{align}
holds for $d$ sufficiently large, thus by (\ref{4.99999}),
\begin{align}\label{4.22222}
\sum_{i=1}^{k}\lambda_{i}^{\frac{1}{2}}\leq \frac{s}{2}(\frac{1}{1-\tau}+\delta)
\end{align}
for $d$ sufficiently large.

Thus from (\ref{4.22222}) and (\ref{5.22223}), there is a positive constant $C_{2}$ depending on $n$, $\xi$, $i_{0}$, $\alpha$ such that
\begin{align}\label{5.33331}
&(1-\xi)\frac{n-1}{n}C_{1}^{\frac{1}{n-1}}(n\alpha)^{-\frac{1}{n-1}}k^{\frac{n}{n-1}}
\leq C_{2}+\frac{s}{2}(\frac{1}{1-\tau}+\delta) \\
=&C_{2}+(\frac{1}{1-\tau}+\delta)\sum_{i=1}^{d}(i+\frac{n}{2}-1)k_{i}.\nonumber
\end{align}

Note that $k_{i}=h_{i}-h_{i-1}$, we have
\begin{align}\label{5.33332}
\sum_{i=1}^{d}(i+\frac{n}{2}-1)k_{i}&=h_{d}(\frac{n}{2}-1)+\sum_{i=1}^{d}i(h_{i}-h_{i-1}) \\ &=h_{d}(\frac{n}{2}-1)+dh_{d}-\sum_{i=1}^{d}h_{i-1}.\nonumber
\end{align}

By (\ref{5.33331}), (\ref{5.33332}) and (\ref{4.33331}), we have
\begin{align}\label{5.11119}
(\frac{1}{1-\tau}+\delta)\sum_{i=1}^{d}h_{i-1}\leq C_{2}+(\frac{1}{1-\tau}+\delta)(\frac{n}{2}-1+d)h_{d} \\ -(1-\xi)\frac{n-1}{n}C_{1}^{\frac{1}{n-1}}(n\alpha)^{-\frac{1}{n-1}} \biggl(\frac{A-1}{A}\biggr)^{\frac{n}{n-1}}h_{d}^{\frac{n}{n-1}}.\nonumber
\end{align}

Suppose $F_{1}$, $F_{2}$ are positive constants.
It is easy to prove that the function $F:\mathbb{R}\rightarrow \mathbb{R}$ defined by $F(h)=F_{1}h-\frac{n-1}{n}F_{2}h^{\frac{n}{n-1}}$ has a unique maximum point $h=\bigl(\frac{F_{1}}{F_{2}}\bigr)^{n-1}$, and the maximum value is $\frac{F_{1}^{n}}{nF_{2}^{n-1}}$.

Take $F_{1}=(\frac{1}{1-\tau}+\delta)(\frac{n}{2}-1+d)$, $F_{2}=(1-\xi)C_{1}^{\frac{1}{n-1}}(n\alpha)^{-\frac{1}{n-1}}\bigl(\frac{A-1}{A}\bigr)^{\frac{n}{n-1}}$ and consider the right hand side of (\ref{5.11119}) as a function of $h_{d}$, we obtain
\begin{align}\label{5.11120}
(\frac{1}{1-\tau}+\delta)\sum_{i=1}^{d}h_{i-1}\leq C_{2}+ \frac{(\frac{1}{1-\tau}+\delta)^{n}(\frac{n}{2}-1+d)^{n}}{(1-\xi)^{n-1}C_{1}\alpha^{-1}\bigl(\frac{A-1}{A}\bigr)^{n}}.
\end{align}

Multiply both sides of (\ref{5.11120}) by $d^{-n}$, and then let $d\rightarrow\infty$, we have

\begin{align}\label{5.11121}
(\frac{1}{1-\tau}+\delta)\limsup_{d\rightarrow\infty}d^{-n}\sum_{i=1}^{d}h_{i-1}\leq \frac{(\frac{1}{1-\tau}+\delta)^{n}}{(1-\xi)^{n-1}C_{1}\alpha^{-1}\bigl(\frac{A-1}{A}\bigr)^{n}}.
\end{align}

Let $\xi\rightarrow0$, $\tau\rightarrow0$, $\delta\rightarrow0$, $A\rightarrow\infty$, we obtain (\ref{5.22225}).

On the other hand, from Theorem \ref{main-2} and Proposition \ref{1.11117}, for any small $\epsilon>0$, there exists $d_{0}$ such that for any $d\geq d_{0}$,
\begin{align}
&h_{d}\geq N_{(X,d_{X},m_{X})}(d(d+n-2)-\epsilon)-1 \\
\geq &N_{(X,d_{X},m_{X})}(d^{2}) \geq d^{n-1}\biggl(\frac{2\alpha}{(n-1)!\omega_{n}}-\epsilon\biggr).\nonumber
\end{align}

Thus
\begin{align}\label{5.22224}
&\sum_{i=1}^{d}h_{i-1}\geq\biggl(\frac{2\alpha}{(n-1)!\omega_{n}}-\epsilon\biggr)\biggl[\sum_{i=1}^{d-1}i^{n-1} -C(d_{0})\biggr]\\
=& \biggl(\frac{2\alpha}{(n-1)!\omega_{n}}-\epsilon\biggr)\biggl[\frac{1}{n}\sum_{j=0}^{n-1}(-1)^{j}
\left(
  \begin{array}{c}
    n \\
    j \\
  \end{array}
\right)
B_{j}(d-1)^{n-j} -C(d_{0})\biggr] \nonumber\\
=& \biggl(\frac{2\alpha}{(n-1)!\omega_{n}}-\epsilon\biggr)\biggl[\frac{1}{n}d^{n}+o(d^{n})\biggr],\nonumber
\end{align}
where we use Faulhaber's formula in the first equality, and $B_{j}$ are Bernoulli numbers.
Multiply both sides of (\ref{5.22224}) by $d^{-n}$, and let $d\rightarrow\infty$, we have
$$\liminf_{d\rightarrow\infty}d^{-n}\sum_{i=1}^{d}h_{i-1}\geq \frac{2\alpha}{n!\omega_{n}}-\frac{\epsilon}{n}.$$
By the arbitrariness of $\epsilon$, and (\ref{5.22225}), we obtain (\ref{1.22221}).

In the next we prove (\ref{1.22222}).

Note that we already have
\begin{align}
\liminf_{d\rightarrow\infty}d^{1-n}h_{d}\geq \frac{2\alpha}{(n-1)!\omega_{n}}\nonumber
\end{align}
by Theorem \ref{main-2} and Proposition \ref{1.11117}.

Assume
$$\liminf_{d\rightarrow\infty}d^{1-n}h_{d}>\frac{2\alpha}{(n-1)!\omega_{n}}.$$
Then there exist $\epsilon> 0$ and $d_{0}\in\mathbb{Z}^{+}$ such that
$$h_{d}>\biggl(\frac{2\alpha}{(n-1)!\omega_{n}}+\epsilon\biggr)d^{n-1}$$
for any $d\geq d_{0}$.
Hence
\begin{align}\label{5.22222}
\sum_{i=1}^{d}h_{i-1}\geq&\biggl(\frac{2\alpha}{(n-1)!\omega_{n}}+\epsilon\biggr)\sum_{i=1}^{d-1} i^{n-1}-C(d_{0},\alpha,\epsilon)\\
=&\biggl(\frac{2\alpha}{(n-1)!\omega_{n}}+\epsilon\biggr)\biggl[\frac{1}{n}d^{n}+o(d^{n})\biggr],\nonumber
\end{align}
where the last equality can be derived from Faulhaber's formula.
Multiply both sides of (\ref{5.22222}) by $d^{-n}$, and let $d\rightarrow\infty$, we have
$$\liminf_{d\rightarrow\infty}d^{-n}\sum_{i=1}^{d}h_{i-1}\geq \frac{2\alpha}{n!\omega_{n}}+\frac{\epsilon}{n},$$
which contradicts to (\ref{1.22221}).

The proof is completed.
\end{proof}

\begin{proof}[Proof of Corollary \ref{cor-1}]
It is well known that (1) and (2) hold on Euclidean space, see Appendix B in \cite{L12} for a detailed proof.
On the other hand, if (1) or (2) holds, then from (\ref{1.22221}), (\ref{1.22222}) and the volume comparison theorem, $(M^{n}, g)$ must be isometric to the Euclidean space.
\end{proof}

\begin{rem}\label{rem5.2}
We remark that, in the proof of Theorem \ref{main-4}, the $i_{0}$ (such that (\ref{5.00002}) holds for any $i\geq i_{0}$) depends on $\xi$ and the cross section $X$.
On the other hand, if we do not assume the tangent cone at infinity is unique, the tangent cone at infinity $C(X)$ depends on $d$ and other parameters used to blow down the manifold, thus in this case we cannot obtain (\ref{5.11121}) directly by letting $d\rightarrow\infty$.
This is one of the reasons why in Theorem \ref{main-4} we assume the tangent cone at infinity is unique.
The other reason is that we use Theorem \ref{main-2}, whose proof also make use of the uniqueness of the tangent cone at infinity.
In fact, it is an open question whether there exists a nontrivial polynomial growth harmonic function on a manifold with nonnegative Ricci curvature and maximal volume growth.
It is also an interesting question whether the conclusion in Theorem \ref{main-4} holds without the assumption on uniqueness of tangent cone at infinity.
\end{rem}

%\section{Further Remarks}\label{s6}

\end{document}